\documentclass[hidelinks,onefignum,onetabnum]{siamart250106}

\definecolor{darkgreen}{rgb}{0.0, 0.55, 0.18}

\newcommand{\mycomment}[1]{%

\ifthenelse{\isodd{\value{page}}}{%
\makeatletter\normalmarginpar%
\marginpar{\tiny {#1}}%
}{%
\makeatletter\reversemarginpar%
\marginpar{\tiny {#1}}%
}}%

\usepackage[T1]{fontenc}      
\usepackage{type1cm}           
\usepackage{newtxtext}         
\usepackage[varvw]{newtxmath}  

\usepackage{graphicx}          
\usepackage{subcaption}        
\usepackage{tikz}              
\usepackage{wrapfig}           
\usepackage{epstopdf}          
\ifpdf
  \DeclareGraphicsExtensions{.eps,.pdf,.png,.jpg}
\else
  \DeclareGraphicsExtensions{.eps}
\fi

\usepackage{amsmath}           
\allowdisplaybreaks[1]

\usepackage{color}             
\usepackage{multicol}          
\usepackage{multirow}          
\usepackage{adjustbox}         
\usepackage[bottom]{footmisc}  

\usepackage{url}               
\usepackage{bm}                
\usepackage{ulem}              
\usepackage{comment}           
\usepackage{siunitx}           
\usepackage{algorithmic}       
\usepackage{makeidx}           
\usepackage{cleveref}          

\newcommand{\dd}{\mathop{}\!\mathrm{d}}

\definecolor{haesslichesblau}{rgb}{0.0,0.32,0.61}

\newcommand{\STAB}[1]{\begin{tabular}{@{}c@{}}#1\end{tabular}}

\newcommand{\condsum}[2]{\!\!\!\sideset{}{_{\raisebox{0.3mm}[0pt][0pt]{$\scriptstyle #1$}}}\sum_{\substack{\raisebox{-0.85mm}[0pt][0pt]{$\scriptstyle\hspace{0.8em}#2$}}}\!\!\!}

\usepackage{acronym}           

\acrodef{BDDC}{Balancing Domain Decomposition by Constraints}
\acrodef{CTW}{Controlled Tensile Weldability}
\acrodef{DDMs}{domain decomposition methods}
\acrodef{DOFs}{degrees of freedom}
\acrodef{FETI-DP}{Finite Element Tearing and Interconnecting -- Dual Primal}
\acrodef{GDSW}{Generalized Dryja-Smith-Widlund}
\acrodef{GMRES}{Generalized Minimal Residual}
\acrodef{LBW}{Laser Beam Welding}
\acrodef{PCG}{Preconditioned Conjugate Gradient}
\acrodef{PDE}{partial differential equation}
\acrodef{PDEs}{partial differential equations}
\acrodef{RGDSW}{Reduced Generalized Dryja-Smith-Widlund}
\acrodef{VEM}{Virtual Element Method}

\newsiamremark{remark}{Remark}
\newsiamremark{hypothesis}{Hypothesis}
\crefname{hypothesis}{Hypothesis}{Hypotheses}

\newsiamthm{claim}{Claim}

\headers{Two-level Schwarz Preconditioners for high-order discretizations in 2D}{T. Bevilacqua, A. Klawonn, M. Lanser and A. Wasiak}

\title{Two-level additive Schwarz preconditioners with GDSW coarse spaces for high-order discretizations in 2D}
\newcommand{\siscfundingnote}{\funding{This project has received funding from the Deutsche Forschungsgemeinschaft (DFG) as part of the Forschungsgruppe (Research Unit) 5134 ``Solidification Cracks During Laser Beam Welding -- High Performance Computing for High Performance Processing'' under DFG project number 434946896.}}

\newcommand{\siscdeptaffil}{\thanks{Department of Mathematics, University of Cologne, Germany}}
\newcommand{\sisccenteraffil}{\!\thanks{Center for Data and Simulation Science, University of Cologne, Germany}}
\newcommand{\siscemailfootnote}{\thanks{\email{E-mail: [tommaso.bevilacqua,axel.klawonn,martin.lanser,adam.wasiak]@uni-koeln.de}}}
\newcommand{\siscdeptmark}{\ensuremath{{}^{*}}}
\newcommand{\sisccentermark}{\ensuremath{\!{}^{\dagger}}}

\author{%
  Tommaso Bevilacqua\siscdeptaffil%
  \and Axel Klawonn\siscdeptmark\sisccenteraffil
  \and Martin Lanser\siscdeptmark\sisccentermark%
  \and Adam Wasiak\siscdeptmark%
  \phantom{\siscemailfootnote}.
}

\ifpdf
\hypersetup{
  pdftitle={},
  pdfauthor={T. Bevilacqua, A. Klawonn, M. Lanser and A. Wasiak}
}
\fi

\begin{document}

\maketitle
\begingroup
\renewcommand{\thefootnote}{}
\footnotetext{\siscfundingnote}
\addtocounter{footnote}{-1}
\endgroup
\begin{abstract}
Two-level overlapping Schwarz preconditioners with Generalized Dryja-Smith-Widlund (GDSW) and Reduced GDSW (RGDSW) coarse spaces for high-order finite element methods (FEM) and virtual element methods (VEM) are investigated and applied to the Poisson problem in 2D. The GDSW coarse space is a well-established robust choice for preconditioners in domain decomposition, while the RGDSW coarse space offers a competitive alternative by reducing the coarse problem dimension while maintaining robustness. On the theoretical side, a condition number bound for both preconditioners is proved for conforming $\mathbb{P}_k$ and $\mathbb{Q}_k$ finite element spaces with explicit dependence on the polynomial degree $k$. Through numerical experiments on structured triangular and quadrilateral meshes for FEM and on polygonal meshes for VEM, it is demonstrated that both GDSW and RGDSW maintain scalability properties across polynomial orders up to $k=9$ for selected degrees. The RGDSW approach achieves a reduction in the coarse problem dimension by a factor of approximately three compared to GDSW while maintaining competitive iteration counts.
\end{abstract}

\begin{keywords}
domain decomposition, overlapping Schwarz preconditioners, GDSW, RGDSW, virtual element method, high-order methods
\end{keywords}

\begin{MSCcodes}
65N55, 65N30, 65F08
\end{MSCcodes}

\section{Introduction}
\label{intro}
High-order discretizations of elliptic partial differential equations offer the well-known advantage of rapid convergence in the polynomial degree $k$, and are widely used in both academic and engineering applications~\cite{karniadakis_spectral_2005,szabo_finite_1991}.
Classical high-order finite element methods (FEM) achieve this on structured triangular and quadrilateral meshes.
The virtual element method (VEM), introduced by Beir\~{a}o da Veiga et al.\ in~\cite{beirao_basic_2013} and extended to high-order in~\cite{ahmad_equivalent_2013, beirao_dassi_russo_highorder_2017}, obtains the same goal on general polygonal and polyhedral meshes.
Rather than constructing explicit basis functions, the VEM defines local spaces implicitly and carries out all computations via projections onto polynomial spaces, exploiting as degrees of freedom (DOFs), i.e., nodal evaluations on vertex and edge nodes, and interior moments.

The solution of large sparse linear systems arising from high-order FEM and VEM discretizations requires efficient and scalable iterative solvers.
Domain decomposition methods, which partition the computational domain into overlapping or non-overlapping subdomains and combine local solvers with a coarse-level correction, are among the most effective approaches for parallel computing~\cite{toselli_widlund_ddm_2005}.
For classical high-order FEM, \ac{BDDC} and \ac{FETI-DP} preconditioners for spectral element discretizations were analyzed in~\cite{klawonn_pavarino_rheinbach_spectral_2008, pavarino_widlund_bddc_feti_spectral_2007, pavarino_widlund_zampini_bddc_spectral_2010}.
For VEM, the corresponding theory is less developed.
Within the class of non-overlapping methods, \ac{BDDC} and \ac{FETI-DP} preconditioners for VEM were analyzed in~\cite{bertoluzza_fetidp_2017} for two-dimensional elliptic problems and extended to three dimensions in~\cite{bertoluzza_fetidp_2020}.
Adaptive \ac{BDDC} preconditioners for high-order VEM discretizations of three-dimensional elliptic problems in mixed form were studied in~\cite{dassi_scacchi_zampini_bddc_2022}.
More recently, \ac{BDDC} preconditioners for VEM discretizations of the three-dimensional Stokes equations were developed and analyzed in~\cite{bevilacqua_bddc_stokes_vem_2024}, with numerical evidence of robustness with respect to the polynomial degree $k$.
Within the class of overlapping Schwarz methods, $h$-dependent condition number estimates for FEM are known; see e.g.~\cite{smith_domain_1991, toselli_widlund_ddm_2005}.
For high-order discretizations, the question of $k$-robustness is more delicate: Pavarino~\cite{pavarino_additive_1994} proved that, for continuous high-order FEM on structured quadrilateral meshes, a $\mathbb{Q}_1$ coarse space suffices to obtain a condition number bound independent of $k$, whereas Antonietti and Houston~\cite{antonietti_class_2011} showed that, for high-order discontinuous Galerkin methods, polynomial enrichment of the coarse space is required.
A robust choice of coarse space for general, possibly irregular domain decompositions is provided by the Generalized Dryja-Smith-Widlund (GDSW) coarse space, introduced by Dohrmann, Klawonn, and Widlund in~\cite{dohrmann_less_regular_2008, dohrmann_family_2008}.
The GDSW coarse space is constructed by decomposing the subdomain interface into vertices, edges, and (in 3D) faces, and associating to each component a coarse basis function obtained as the discrete harmonic extension of a local partition-of-unity function.
The Reduced GDSW (RGDSW) variant~\cite{dohrmann_design_2017} reduces the coarse problem dimension while maintaining robustness.

For virtual element methods, overlapping Schwarz methods have been studied in two dimensions by Calvo~\cite{calvo_approximation_2018, calvo_overlapping_2019}, with extensions to high-contrast problems~\cite{calvo_overlapping_2024} and irregular subdomain boundaries~\cite{calvo_virtual_2020}, but all these works are restricted to the lowest-order case $k=1$.
To the best of our knowledge, no analysis of GDSW-type overlapping Schwarz preconditioners for high-order VEM discretizations exists in the literature.

In this work, we study two-level additive Schwarz preconditioners with GDSW and RGDSW coarse spaces for the Poisson problem discretized by both high-order FEM and VEM in two dimensions.
On the theoretical side, we derive, to the best of our knowledge for the first time, a condition number bound for the GDSW and RGDSW preconditioners applied to conforming $\mathbb{P}_k$ and $\mathbb{Q}_k$ finite element spaces with explicit dependence on the polynomial degree $k$
\begin{equation*}
  \kappa(M^{-1}A) \leq C\!\left(1 + \frac{H}{\delta}\right)\!\left(1 + \log\!\left(\frac{H}{h}\right)\right)\!\left(1 + \log\!\left(\frac{H k^2}{h}\right)\right),
\end{equation*}
where $H$ is the maximum subdomain diameter, $h$ is the maximum element diameter, $\delta$ is the overlap width, and $M$ and $A$ are the preconditioner and stiffness matrix, respectively.
On the computational side, we demonstrate scalability and $k$-robustness of GDSW and RGDSW for both FEM and VEM, up to $k=9$, on structured quadrilateral and triangular meshes for both discretizations and on polygonal meshes for VEM.
The experiments use an in-house \textsc{Matlab} implementation.

Preliminary results were presented in~\cite{bevilacqua_parallel_schwarz_2025}.
The article is organized as follows. Section~\ref{sec:discretizations} introduces the model problem and the FEM and VEM discretizations. In \cref{sec:GDSW}~the GDSW and RGDSW preconditioners and their construction for high-order discretizations are described. The condition number bound is derived in \cref{sec:condbound}. Numerical results are presented in \cref{sec:results}.
\section{Model problem}
\label{sec:discretizations}

The continuous problem used throughout the paper is the two-dimensional Poisson equation on a bounded polygonal domain $\Omega$ with prescribed Dirichlet data $g\in H^{1/2}(\partial\Omega)$.
We reduce the problem to one with homogeneous boundary values by writing the solution as $u=\tilde{u}+w$, where $\tilde{u}$ is a harmonic extension of $g$ and $w\in H^1_0(\Omega)$.
For $f\in L^2(\Omega)$, the variational problem is given by
\[
a(w,v)=F(v)\qquad \forall v\in H^1_0(\Omega),
\]
where
\[
a(w,v)=\int_\Omega \nabla w\cdot\nabla v\,\dd x,\qquad
F(v)=\int_\Omega f v\,\dd x-\int_\Omega \nabla\tilde{u}\cdot\nabla v\,\dd x.
\]

\subsection{High-order finite element discretization}
\label{sec:fem_discretization}
 We introduce a shape-regular triangular tessellation $\mathcal{T}_h$ and a shape-regular quadrilateral tessellation $\mathcal{Q}_h$ of $\Omega$ with maximum element diameter $h$. This gives rise to two families of local approximation spaces.

\paragraph{Triangular meshes}
On each triangle $T \in \mathcal{T}_h$, the local space is the polynomial space
\begin{equation*}
\mathbb{P}_k(T) := \operatorname{span}\bigl\{x_1^{i_1} x_2^{i_2} : i_1 + i_2 \leq k\bigr\},
\end{equation*}
consisting of all polynomials of total degree of at most $k$. The \ac{DOFs} are nodal evaluations at:
\begin{itemize}
    \item the 3 vertices of $T$,
    \item for $k \geq 2$, values at $k-1$ equidistant interior points on each of the 3 edges, and,
    \item for $k \geq 3$, values at $\binom{k-1}{2}$ interior points within $T$.
\end{itemize}
The total number of \ac{DOFs} per triangle is $\binom{k+2}{2} = \dim (\mathbb{P}_k(T))$.
The stiffness matrix entries are computed via a symmetric quadrature rule on the reference triangle with an exactness degree of at least $2(k-1)$~\cite{schwab_p_hp_fem_1998}.

\paragraph{Quadrilateral meshes}
On each quadrilateral $Q \in \mathcal{Q}_h$, the local space is the tensor-product polynomial space
\begin{equation*}
\mathbb{Q}_k(Q) := \operatorname{span}\bigl\{x_1^{i_1} x_2^{i_2} : 0 \leq i_1, i_2 \leq k\bigr\}.
\end{equation*}
The \ac{DOFs} are nodal evaluations at:
\begin{itemize}
    \item the 4 vertices of $Q$,
    \item for $k \geq 2$, values at $k-1$ equidistant interior points on each of the 4 edges, and,
    \item for $k \geq 2$, values at $(k-1)^2$ equidistant interior points within $Q$.
\end{itemize}
The stiffness matrix entries are computed via Gauss--Legendre quadrature with $k+1$ points per direction, which integrates them exactly on affine elements.

\paragraph{Global spaces and discrete problem}
The corresponding global conforming finite element spaces are
\begin{align*}
V_{h,\mathcal{T}}^k &:= \bigl\{ v_h \in H^1_0(\Omega) : v_h|_{T} \hspace{0.45mm}\in\hspace{0.15mm}\, \mathbb{P}_k(T) \;\; \,\forall T \hspace{0.3mm}\in\hspace{0.3mm} \mathcal{T}_h \,\bigr\}, \\
V_{h,\mathcal{Q}}^k &:= \bigl\{ v_h \in H^1_0(\Omega) : v_h|_{Q} \in \mathbb{Q}_k(Q) \;\; \forall Q \in \mathcal{Q}_h \bigr\}.
\end{align*}
Regardless of the mesh type, the FEM approximation $u_h \in V_{h}^k$ satisfies
\begin{equation*}
a(u_h,v_h) = F(v_h) \quad \forall v_h \in V_{h}^k.
\end{equation*}
Since all \ac{DOFs} are nodal evaluations, the GDSW coarse space construction described in \Cref{sec:GDSW} applies directly in both cases.
For the theory of high-order Lagrange finite elements we refer to~\cite{brenner_mathematical_2008, ciarlet_finite_1978, schwab_p_hp_fem_1998}.

\subsection{Virtual element discretization}
\label{sec:vem_discretization}

The virtual element method extends finite element techniques to general polygonal meshes. Given a polygonal tessellation $\mathcal{T}_h$ of $\Omega$ with maximum element diameter $h$, we follow the enhanced VEM framework in \cite{ahmad_equivalent_2013} and define the local virtual element spaces and their \ac{DOFs}. We denote by $\mathcal{M}_\ell(S) := \{((\mathbf{x}-\mathbf{x}_S)/h_S)^{\mathbf{s}} : |\mathbf{s}| \leq \ell\}$ the space of scaled monomials of degree at most $\ell$ on the geometric object $S$ (edge, face, or element), with centroid $\mathbf{x}_S$ and diameter $h_S$, and by $\mathcal{M}^*_\ell(S)$ the subset of monomials of degree exactly $\ell$.

For each polygon $K$ with $n_K$ vertices, we first define the enlarged space
\begin{equation*}
\widetilde{V}_h^k(K) := \Big\{ v_h \in H^1(K) : \, v_h|_{\partial K} \in \mathbb{B}_k(\partial K), \; \Delta v_h \in \mathbb{P}_k(K) \Big\},
\end{equation*}
where $\mathbb{B}_k(\partial K)$ denotes the space of continuous functions on $\partial K$ that are polynomials of degree at most $k$ on each edge and $\Delta$ is the standard Laplace operator. The local enhanced virtual element space of order $k \geq 1$ is then the subspace
\begin{equation}\label{eq:vem_space_2d}
V_h^k(K) := \Big\{ v_h \in \widetilde{V}_h^k(K) : \, (v_h - \Pi^\nabla_k v_h, \, q^*)_{0,K} = 0 \;\; \forall q^* \in \mathcal{M}^*_{k-1}(K) \cup \mathcal{M}^*_k(K) \Big\},
\end{equation}
where $\Pi^\nabla_k : \widetilde{V}_h^k(K) \to \mathbb{P}_k(K)$ is the elliptic projector defined by~\cite{beirao_dassi_russo_highorder_2017}
\begin{equation*}
\int_K \nabla(v_h - \Pi^\nabla_k v_h) \cdot \nabla p \, \dd x = 0 \quad \forall p \in \mathbb{P}_k(K),
\end{equation*}
with  $\sum_{i} (v_h(\nu_i) - \Pi^\nabla_k v_h(\nu_i)) = 0$ for $k = 1$, where $\nu_i$, $i=1,...,n_K$ are the vertices of $K$, and $\int_K (v_h - \Pi^\nabla_k v_h) \, \dd x = 0$ for $k \geq 2$. This space satisfies $\mathbb{P}_k(K) \subseteq V_h^k(K)$ and has dimension $n_K k + k(k-1)/2$. Following~\cite{beirao_dassi_russo_highorder_2017}, the \ac{DOFs} for $V_h^k(K)$ are:
\begin{itemize}
    \item values of $v_h$ at the $n_K$ vertices of $K$,
    \item for $k \geq 2$, values of $v_h$ at the $k-1$ internal Gauss--Lobatto points on each edge of $K$,
    \item for $k \geq 2$, interior moments:
    $\frac{1}{|K|}\int_K q \, v_h \, \dd x \quad \text{for all } q \in \mathcal{M}_{k-2}(K)$.
\end{itemize}
The orthogonality constraint in \eqref{eq:vem_space_2d} ensures that the moments of order $k-1$ and $k$ of $v_h$ coincide with those of $\Pi^\nabla_k v_h$.

The $L^2$-projection $\Pi^0_k : V_h^k(K) \to \mathbb{P}_k(K)$ is defined by~\cite{beirao_dassi_russo_highorder_2017}
\begin{equation*}
\int_K (v_h - \Pi^0_k v_h) \, q \, \dd x = 0 \quad \forall q \in \mathbb{P}_k(K).
\end{equation*}
With the enhanced formulation of the VEM spaces it is ensured that the projectors $\Pi^\nabla_k$ and $\Pi^0_k$ are computable from the \ac{DOFs} alone.
With these projectors, the discrete bilinear form is constructed as 
\begin{equation*}
a_h(v_h,w_h) = a(\Pi^\nabla_k v_h, \Pi^\nabla_k w_h) + s_h(v_h - \Pi^\nabla_k v_h, w_h - \Pi^\nabla_k w_h),
\end{equation*}
where $s_h(\cdot,\cdot)$ is the d-recipe stability term \cite{mascotto_ill-conditioning_2018} which is computable from the \ac{DOFs}. The right-hand side term is given by $F_h(v_h) = F(\Pi^0_k v_h)$.

The VEM approximation $u_h \in V_h \subset H^1_0(\Omega)$ satisfies
\begin{equation*}
a_h(u_h,v_h) = F_h(v_h) \quad \forall v_h \in V_h,
\end{equation*}
where $a_h(\cdot,\cdot)$ and $F_h(\cdot)$ are consistent and stable approximations of $a(\cdot,\cdot)$ and $F(\cdot)$. For further details on mesh assumptions, stability analysis, and convergence properties, we refer to \cite{ahmad_equivalent_2013,beirao_basic_2013,beirao_stability_2017}.

\begin{remark}
  Since we use high-order spaces, we orthogonalize the basis functions of the local spaces $V_h^k(K)$ with respect to the $L^2$-inner product, for both finite and virtual elements. We use an orthonormal modal polynomial construction (and nodal reconstruction) for FEM and the methods described in \cite[Appendix A]{mascotto_ill-conditioning_2018} for VEM. This leads to better conditioning of the local stiffness matrices and improved numerical stability, especially for high polynomial degrees.
\end{remark}

\section{Overlapping Schwarz Method with GDSW coarse spaces}\label{sec:GDSW}

We use a two-level overlapping Schwarz preconditioner whose global correction is spanned by GDSW-type coarse functions~\cite{dohrmann_less_regular_2008, dohrmann_family_2008}.
In the notation of the assembled stiffness matrix, the preconditioner has the form
\[
M_{\mathrm{GDSW}}^{-1}
=\Phi A_0^{-1}\Phi^T+\sum_{i=1}^N R_i^TA_i^{-1}R_i.
\]
Here $R_i$ restricts to the overlapping subdomain $\Omega_i$, $A_i=R_iAR_i^T$, and the coarse matrix is $A_0=\Phi^TA\Phi$.
The difference between GDSW and RGDSW lies only in the choice of the columns of $\Phi$, that correspond to the coarse space basis functions.

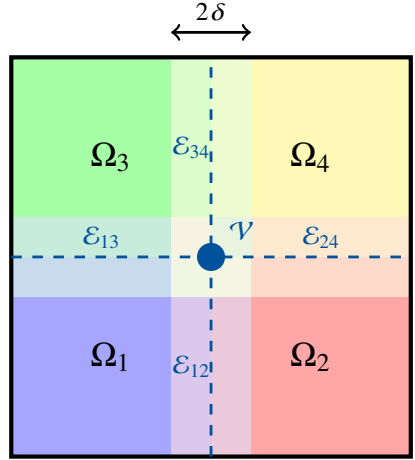
\begin{wrapfigure}{r}{0.43\textwidth}
\vspace{-0.4\baselineskip}
\centering
\resizebox{0.41\textwidth}{!}{%
\begin{tikzpicture}[scale=1]
  \fill[blue!35] (0,0) rectangle (2,2);
  \fill[red!35] (2,0) rectangle (4,2);
  \fill[green!35] (0,2) rectangle (2,4);
  \fill[yellow!35] (2,2) rectangle (4,4);

  \fill[blue!12, opacity=0.6] (2,0) rectangle (2.4,2);
  \fill[blue!12, opacity=0.6] (0,2) rectangle (2,2.4);
  \fill[blue!12, opacity=0.6] (2,2) rectangle (2.4,2.4);

  \fill[red!12, opacity=0.6] (1.6,0) rectangle (2,2);
  \fill[red!12, opacity=0.6] (2,2) rectangle (4,2.4);
  \fill[red!12, opacity=0.6] (1.6,2) rectangle (2,2.4);

  \fill[green!12, opacity=0.6] (2,2) rectangle (2.4,4);
  \fill[green!12, opacity=0.6] (0,1.6) rectangle (2,2);
  \fill[green!12, opacity=0.6] (2,1.6) rectangle (2.4,2);

  \fill[yellow!12, opacity=0.6] (1.6,2) rectangle (2,4);
  \fill[yellow!12, opacity=0.6] (2,1.6) rectangle (4,2);
  \fill[yellow!12, opacity=0.6] (1.6,1.6) rectangle (2,2);

  \draw[very thick] (0,0) rectangle (4,4);

  \draw[thick, dashed, haesslichesblau] (2,0) -- (2,4);
  \draw[thick, dashed, haesslichesblau] (0,2) -- (4,2);

  \node[haesslichesblau] at (1.8, 0.9) {\scriptsize $\mathcal{E}_{12}$};
  \node[haesslichesblau] at (0.9, 2.2) {\scriptsize $\mathcal{E}_{13}$};
  \node[haesslichesblau] at (3.1, 2.2) {\scriptsize $\mathcal{E}_{24}$};
  \node[haesslichesblau] at (1.8, 3.1) {\scriptsize $\mathcal{E}_{34}$};

  \fill[haesslichesblau] (2,2) circle (4pt);
  \node[haesslichesblau, above right] at (2.05,2.05) {\scriptsize $\mathcal{V}$};

  \node at (1,1) {$\Omega_1$};
  \node at (3,1) {$\Omega_2$};
  \node at (1,3) {$\Omega_3$};
  \node at (3,3) {$\Omega_4$};

  \draw[<->, thick] (1.6,4.25) -- (2.4,4.25);
  \node[above] at (2,4.25) {\scriptsize $2\delta$};

\end{tikzpicture}%
}
\caption{Overlapping subdomains and interface components in two dimensions.}
\label{fig:gdsw-overlap-interface}
\vspace{-1.3\baselineskip}
\end{wrapfigure}
In order to construct them, we need to decompose the interface \ac{DOFs} of a non-overlapping decomposition into vertices $\mathcal{V}$ and edges $\mathcal{E}$, as sketched in \Cref{fig:gdsw-overlap-interface}, and define, for each component $\mathcal{P} \in \{\mathcal{V}, \mathcal{E}\}$, the interface function $\phi_{\mathcal{P}}$ with appropriate values on the \ac{DOFs} belonging to $\mathcal{P}$ and zero on the rest of the interface values. We note, that we will denote an edge shared by subdomains $\Omega_i$ and $\Omega_j$ as $\mathcal{E}_{ij}$, as shown in \Cref{fig:gdsw-overlap-interface}. Denoting with $A_{II}$ the submatrix of $A$ corresponding to interior \ac{DOFs}, and $A_{I\Gamma}$ the submatrix coupling interior to interface \ac{DOFs} (denoted by $\Gamma$), the interior part of the coarse function is then given by $\phi_I^{\mathcal{P}} = -A_{II}^{-1} A^{\vphantom{^{-1}}}_{I\Gamma} \phi_{\mathcal{P}}$, so that the coarse basis function associated with $\mathcal{P}$ is $\phi^{\mathcal{P}} = (\phi_I^{\mathcal{P}}, \phi_{\mathcal{P}})$. We note that here, each $\phi^{\mathcal{P}}$ is a vector in $\mathbb{R}^n$, where $n = \dim V_h^k$ is the number of \ac{DOFs} on the mesh $\mathcal{T}_h$; no separate coarse mesh is introduced. Representative coarse basis functions are illustrated in \Cref{fig:gdsw_rgdsw_3d_inner_modes}. For the discretization used in this work, all \ac{DOFs} on the interface are nodal evaluations, therefore the interface function $\phi_{\mathcal{P}}$ is simply set to 1 on all nodes belonging to $\mathcal{P}$ and 0 on the other interface components, at least for the standard version of GDSW. Denoting by $|\mathcal{V}|$ and $|\mathcal{E}|$ the number of vertices and edges of the non-overlapping interface decomposition, the standard GDSW coarse space is thus $\Phi \in \mathbb{R}^{n \times N_c}$, with $N_c = |\mathcal{V}| + |\mathcal{E}|$ for GDSW (one basis function per subdomain vertex and per subdomain edge) and $N_c = |\mathcal{V}|$ for RGDSW, and coarse matrix $A_0 = \Phi^T A \Phi \in \mathbb{R}^{N_c \times N_c}$.

\begin{figure}[b!]
\centering
\begin{subfigure}[t]{0.32\textwidth}
  \centering
  \includegraphics[width=\textwidth]{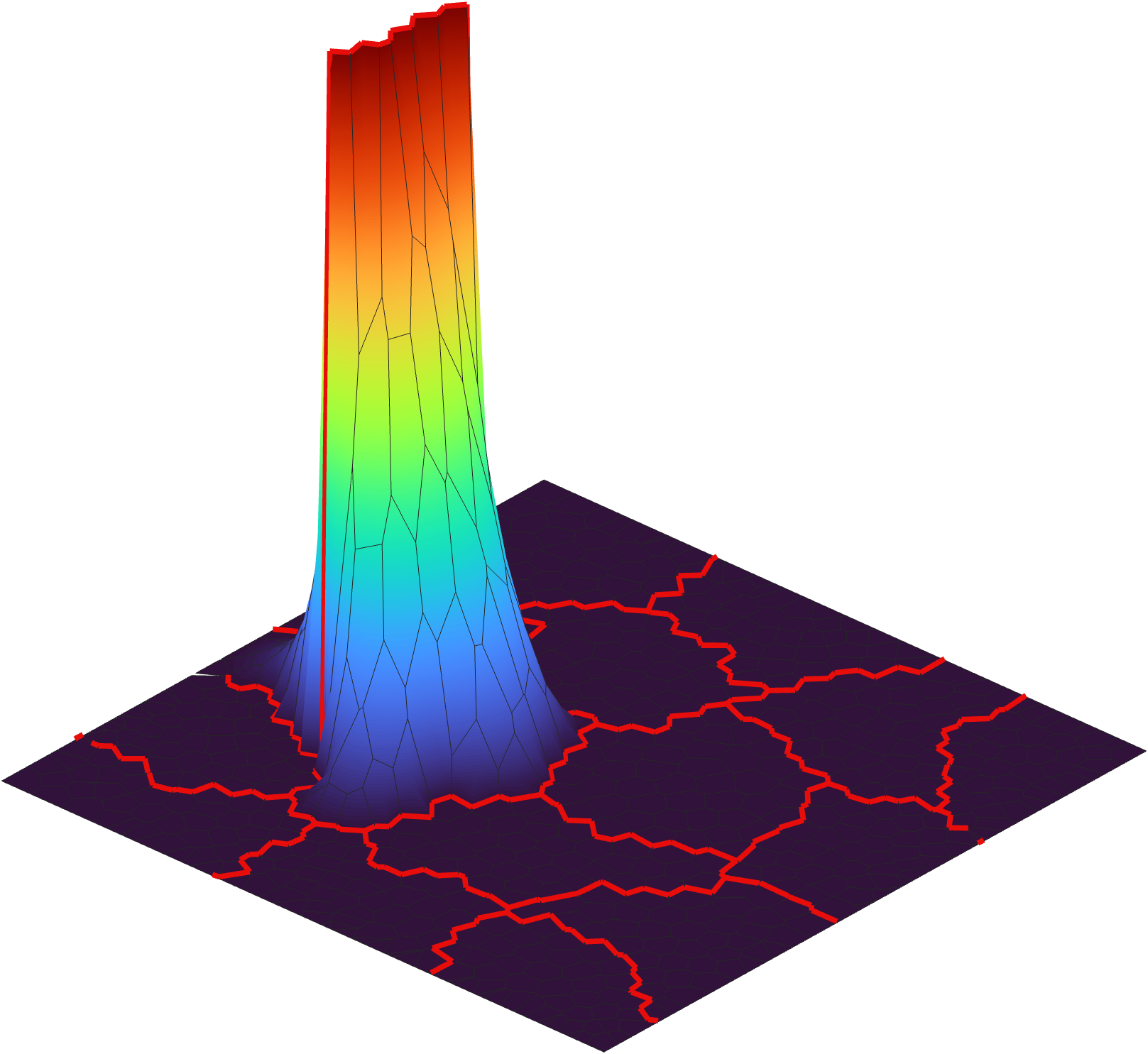}
  \caption{GDSW edge mode}
\end{subfigure}\hfill
\begin{subfigure}[t]{0.32\textwidth}
  \centering
  \includegraphics[width=\textwidth]{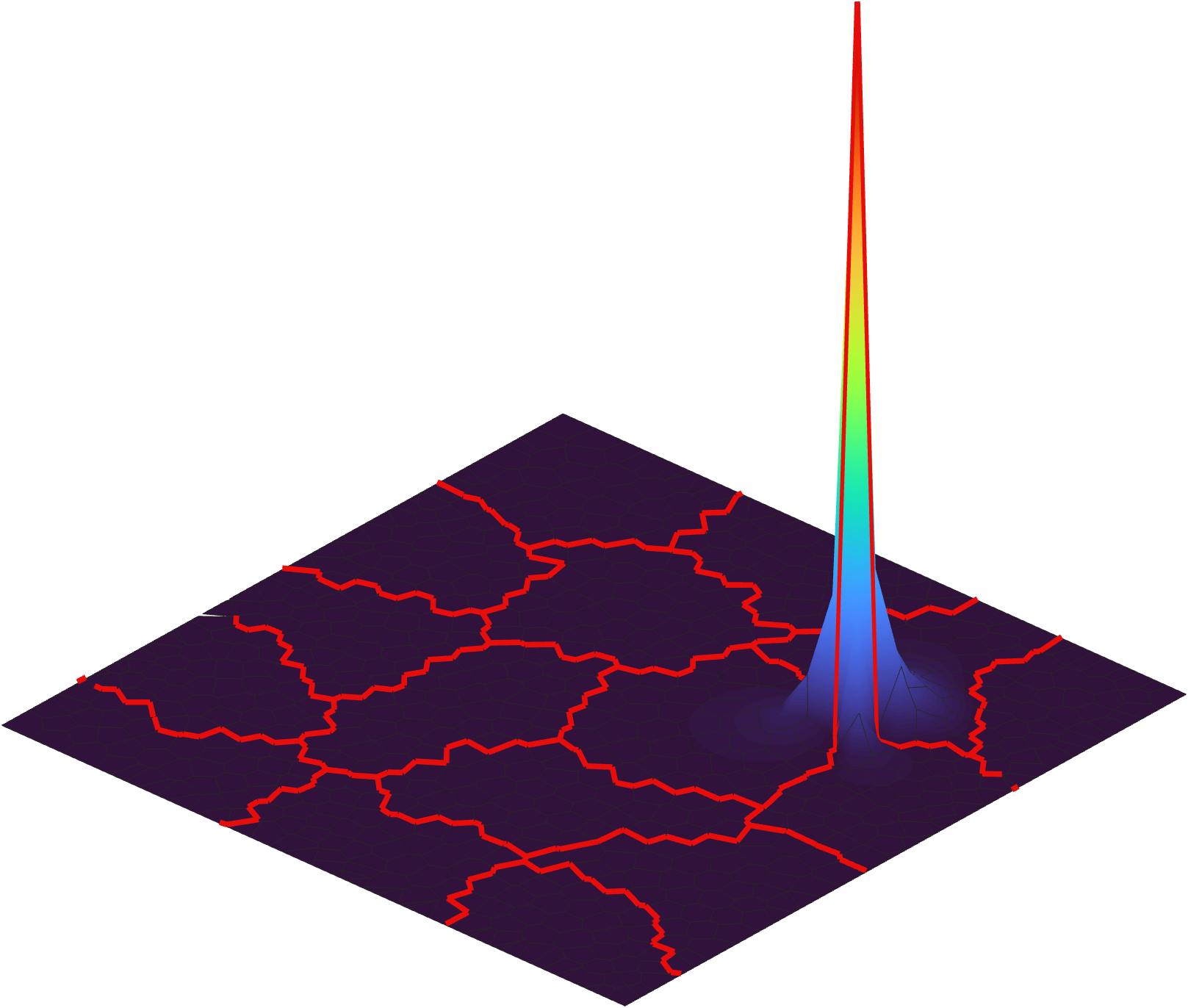}
  \caption{GDSW vertex mode}
\end{subfigure}\hfill
\begin{subfigure}[t]{0.32\textwidth}
  \centering
  \includegraphics[width=\textwidth]{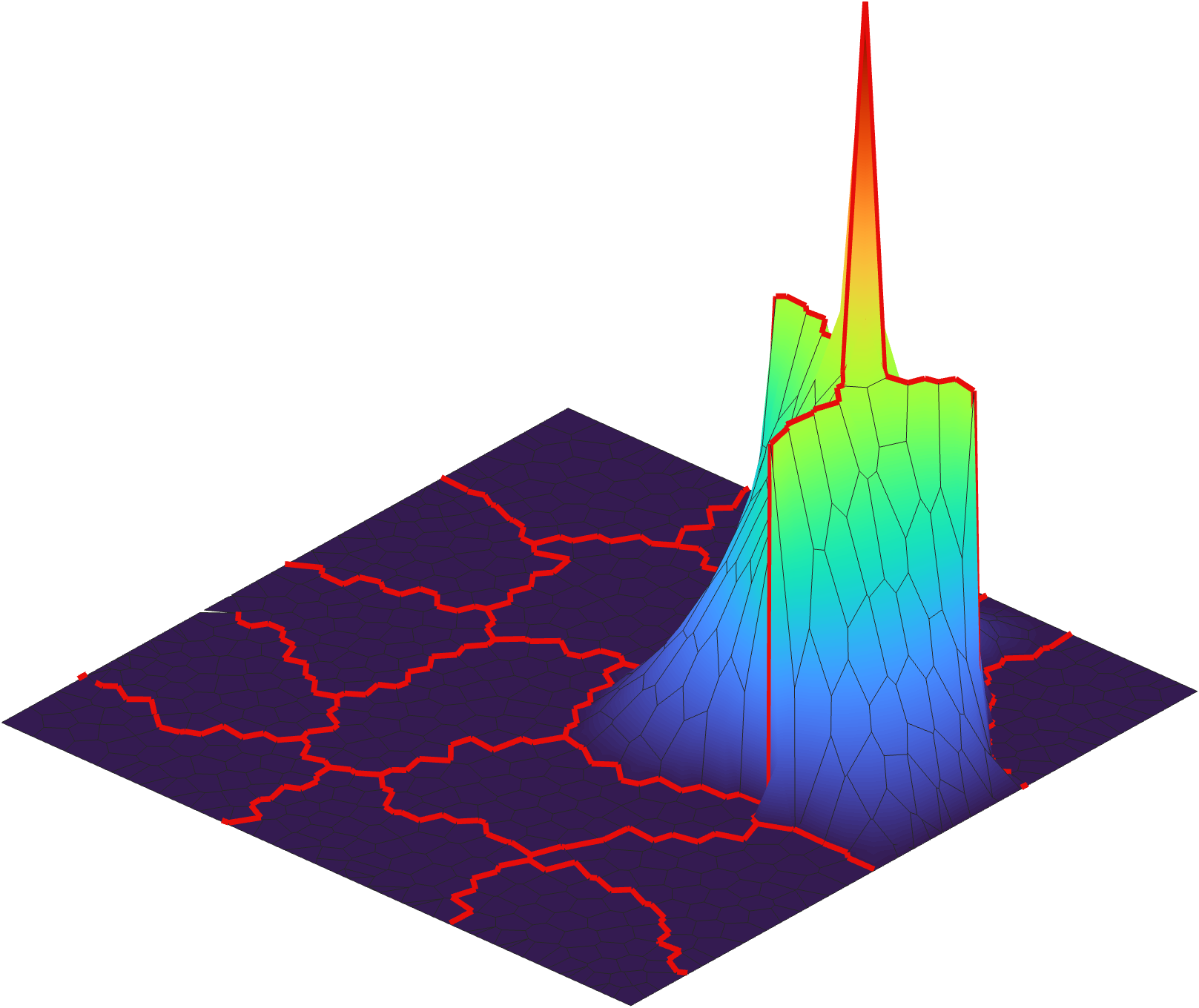}
  \caption{RGDSW mode}
\end{subfigure}
\vspace{-4mm}
\caption{{Representative coarse basis functions for GDSW (edge and vertex) and RGDSW on a polygonal mesh and a fixed METIS domain decomposition.}}\vspace{-9mm}
\label{fig:gdsw_rgdsw_3d_inner_modes}
\end{figure}

\subsection{RGDSW}

The Reduced GDSW (RGDSW) coarse space~\cite{dohrmann_design_2017} is a variant of the standard GDSW coarse space, designed to reduce the dimension of the coarse problem while maintaining robust convergence properties.
For each vertex $\mathbf{v}$ of the domain decomposition, we define the corresponding RGDSW basis function $\phi^{\mathcal{P}_{\mathbf{v}}}$ as follows. Let $\mathcal{E}_{\mathbf{v}}$ denote the (open) set of all edge \ac{DOFs} connected to $\mathbf{v}$ across different subdomains and let $n_{\mathcal{E}_\mathbf{v}}$ be the number of non-Dirichlet vertices adjacent to the edge.
The basis function $\phi^{\mathcal{P}_{\mathbf{v}}}$ is constructed by first prescribing the following values on the interface \ac{DOFs} using
\begin{equation*}
\mathcal{P}_{\mathbf{v}}:= \{\mathbf{v}\}\cup \mathcal{E}_{\mathbf{v}},\qquad\text{and}\qquad\phi_{\mathcal{P}_{\mathbf{v}}}(x) := \begin{cases}
1 & x=\mathbf{v}, \\
\tfrac{1}{n_{\mathcal{E}_\mathbf{v}}} & x \in  \mathcal{E}_{\mathbf{v}}, \\
0 & \text{else on the interface values}.
\end{cases}
\end{equation*}
Once these interface values are specified, the interior part of the basis function $\phi^{\mathcal{P}_{\mathbf{v}}}$, denoted by $\phi^{\mathcal{P}_{\mathbf{v}}}_I$ is computed by solving a local problem analogous to the GDSW case. The full basis function is then given by $\phi^{\mathcal{P}_{\mathbf{v}}} = (\phi^{\mathcal{P}_{\mathbf{v}}}_{I}, \phi_{\mathcal{P}_{\mathbf{v}}})$. Correspondingly, $M_\mathrm{RGDSW}^{-1}$ is defined by using the RGDSW basis functions in the coarse correction.

\subsection{GDSW and RGDSW for high-order discretizations}
\label{sec:gdsw_high_order}

The coarse basis functions are defined on the subdomain interface and then harmonically extended into the subdomain interiors by solving local Dirichlet problems, as described in \Cref{sec:GDSW}. For the Poisson problem, the coarse space must contain the constant function, since constants span the kernel of the gradient operator and represent the lowest-energy modes of the bilinear form $a(\cdot,\cdot)$. The standard GDSW and RGDSW constructions achieve this by setting interface \ac{DOFs} values to 1 on the relevant interface component and to 0 elsewhere.

In 2D, both high-order FEM and VEM admit the standard construction without any modifications. For FEM, all \ac{DOFs} are nodal evaluations at vertices and at interior points of edges, so assigning the value 1 to the \ac{DOFs} on a given edge or vertex component leads to the sum of the vertex and edge contributions to represent the constant function. For VEM, the interface \ac{DOFs} consist of vertex evaluations and pointwise evaluations on edges (see \Cref{sec:vem_discretization}); the interior element moments do not appear on the subdomain interface. Therefore, we apply the same rule, that is, interface \ac{DOFs} on the relevant component are set to 1, all others to 0.

\section{Condition number bound for GDSW and RGDSW}
\label{sec:condbound}
In this section we derive the condition number bound for the GDSW and RGDSW preconditioner for conforming finite element spaces of order $k$, including both $\mathbb{P}_k$ and $\mathbb{Q}_k$ spaces. The bound extends to RGDSW via a decomposition of RGDSW basis functions into GDSW components, as detailed at the end of the proof.
\begin{theorem}\label{thm:condbound}
  For $M^{-1}\in\{M_\mathrm{GDSW}^{-1}, M_\mathrm{RGDSW}^{-1}\}$,
    there exists a constant $C > 0$, depending only on the shape regularity of the mesh and independent of $h$, $H$, $k$, and $\delta$, such that
    \begin{equation*}
        \kappa(M^{-1}A) \leq C \Bigl(1 + \tfrac{H}{\delta}\Bigr)\Bigl(1+\log\Bigl(\tfrac{H}{h}\Bigr)\Bigr)\Bigl(1+\log\Bigl(\tfrac{Hk^2}{h}\Bigr)\Bigr),
    \end{equation*}
    where $k$ is the polynomial degree, $H$ is the maximum subdomain diameter, $h$ is the maximum element diameter, and $\delta$ is the overlap width.
\end{theorem}
As far as we are aware, this is the first GDSW condition number bound with explicit $k$-dependence in the logarithmic factor.
\begin{remark}
  The proof of \Cref{thm:condbound} can be extended to the scalar elliptic problem with bilinear form $a(w,v)=\int_\Omega \rho\,\nabla w\cdot\nabla v\,\dd x$, where the coefficient $\rho > 0$ is piecewise constant or at least only slowly varying with respect to the non-overlapping decomposition. In this case, all estimates are carried out subdomain-wise, and the bound in \Cref{thm:condbound} holds with a constant $C$ that is independent of the values of $\rho$; cf.~\cite{dohrmann_less_regular_2008} for the corresponding result for linear finite elements.
\end{remark}
\subsection{Auxiliary results}
We begin by giving some basic definitions. For a domain $\Omega$, with diameter $H$, we use
\begin{align*}
  \|v\|^r_{L^r(\Omega)} &:= \int_\Omega |v|^r\,\dd x,\qquad 1\le r<\infty,\\
  |v|^2_{H^1(\Omega)} &:= \int_\Omega |\nabla v|^2\,\dd x,\\
  \|v\|^2_{H^1(\Omega)} &:= H^{-2}\|v\|^2_{L^2(\Omega)} + |v|^2_{H^1(\Omega)}, \quad {\rm and}\\
  \|v\|_{L^\infty(\Omega)} &:= \operatorname*{ess\,sup}_{x\in \Omega} |v(x)|.
\end{align*}

In order to prove the condition number bound for both preconditioners, we need generalized versions of the discrete Sobolev inequality, appropriate interpolation operators as in \cite[Lemma 4.31]{toselli_widlund_ddm_2005}, and the edge lemma, to elementwise high-order finite element spaces. We prove these here; for this we will need a Poincar\'e inequality~\cite{Hurri1990} throughout. We recall the definition of a John domain \cite[Section 3]{hajlasz_sobolev_john_2001}.

\begin{definition}[John domain]
  A bounded domain $\Omega \subset \mathbb{R}^2$ is called a John domain if there exists a constant $C_J \ge 1$ and a distinguished point $x_0 \in \Omega$ such that every point $x \in \Omega$ can be connected to $x_0$ by a rectifiable curve $\gamma : [0,1] \to \Omega$ with $\gamma(0) = x_0$, $\gamma(1) = x$, and
  \begin{equation*}
    |x-\gamma(t)| \le C_J\operatorname{dist}(\gamma(t), \partial \Omega) ,
  \end{equation*}
  for all $t \in [0,1]$.
\end{definition}

\begin{lemma}[Poincar\'e's inequality]\label{lem:poincare}
Let $\Omega\subset\mathbb R^2$ be a John domain. Then there exists a constant $C_{\mathrm{PF}}>0$, depending only on the shape of $\Omega$, such that
\begin{equation*}
  \|v\|^2_{L^2(\Omega)} \le C_{\mathrm{PF}} |\Omega| |v|^2_{H^1(\Omega)}
\end{equation*}
for all $v\in H^1(\Omega)$ with either zero mean on $\Omega$ or zero trace on a boundary portion of positive measure.
\end{lemma}

\subsection{Discrete Sobolev inequality for high-order FEM}

The discrete Sobolev inequality for the high-order FEM case is shown in \cite{guo_cao_preconditioner_1996,oden_patra_feng_parallel_1997}. To keep the presentation self-contained, we give a different proof that extends the approach of \cite{toselli_widlund_ddm_2005} to the case of high-order finite elements, 
by proving a generalization of an inverse inequality that is used in the analysis of $hp$-DG and $hp$-FEM methods (e.g. \cite[Lemma 6]{cangiani_hp_dg_2017}) to a more general $L^r$-space. We note, that the discrete Sobolev inequality for VEM was shown for degree $k=1$ in \cite[Lemma 5.2]{calvo_approximation_2018}.
\begin{lemma}\label{lem:infty-Lr-simplex}
    Let $\mathcal{T}_h$ be a triangular or quadrilateral mesh satisfying the usual mesh regularity and uniformity assumptions, $K\in \mathcal{T}_h$, and $r\in[1,\infty)$. Then, for all $v\in\mathbb P_k(K)$ or $v\in\mathbb Q_k(K)$, respectively,
    \begin{equation*}
    \|v\|_{L^\infty(K)} \;\le\; C\,\left(\tfrac{k^2}{h_K}\right)^{\tfrac{2}{r}}\,\|v\|_{L^r(K)}.
    \end{equation*}
    \end{lemma}
    
\begin{proof}
Let $\hat{K}$ be the reference triangle or square, $v$ be a function on a physical element $K$, and $F_K: \hat{K} \to K$ be the transformation between the reference element and a physical element with $\hat{v}(\hat{x}) = v(F_K(\hat{x}))$. Using the transformation theorem, we find
\begin{equation*}
  \|\hat{v}\|_{L^r(\hat{K})} \le C h_K^{-
\tfrac{2}{r}} \|{v}\|_{L^r({K})}
\end{equation*}
where $C=1$ in the case of a triangle and $C=2$ in the case of a square; see \cite[Section 13.5]{ern_guermond_fem1_2021}.
Furthermore, we have 
\begin{equation}\label{eq:schwab-inequality}
  \|D^1\hat{v}\|_{L^\infty(\hat{K})} \le \bar{C} k^2\|\hat{v}\|_{L^\infty(\hat{K})}
\end{equation}
from \cite[Theorem 4.76]{schwab_p_hp_fem_1998}, for both triangles and squares with $\bar{C}\ge 0$ independent of $k$. 
Let $\hat{v}\in V_h^k(\hat{K})$ and let $\bar{x}\in \hat{K}$ be the point for which $|\hat{v}(\bar{x})| = \|\hat{v}\|_{L^\infty(\hat{K})}$. If $\|\hat{v}\|_{L^\infty(\hat{K})}=0$, there is nothing to prove. Otherwise, define $\hat{w}:=\operatorname{sign}(\hat{v}(\bar{x}))\hat{v}$. Then $\hat{w}(\bar{x})=\|\hat{v}\|_{L^\infty(\hat{K})}$, $\|\hat{w}\|_{L^\infty(\hat{K})}=\|\hat{v}\|_{L^\infty(\hat{K})}$, and $\|D^1\hat{w}\|_{L^\infty(\hat{K})}=\|D^1\hat{v}\|_{L^\infty(\hat{K})}$.

Now, for any $\hat{x}\in \hat{B}_\varepsilon(\bar{x})$, we have that
\begin{align*}
  \hat{w}(\hat{x}) &\ge \hat{w}(\bar{x}) - \|\hat{x} - \bar{x} \|\cdot \|D^1\hat{w}\|_{L^\infty(\hat{K})} \\
  & \ge \|\hat{v}\|_{L^\infty(\hat{K})} - \varepsilon \cdot \|D^1\hat{v}\|_{L^\infty(\hat{K})}
\end{align*}
In order for the last term to be greater than $\tfrac{1}{2} \|\hat{v}\|_{L^\infty(\hat{K})}$ we need to have that 
\begin{equation*}
  \varepsilon \cdot \|D^1\hat{v}\|_{L^\infty(\hat{K})} < \tfrac{1}{2} \|\hat{v}\|_{L^\infty(\hat{K})}
\end{equation*}
From \eqref{eq:schwab-inequality}, we know that we can choose $\varepsilon$ such that $\frac{1}{2\varepsilon} = \bar{C} k^2 \iff \varepsilon = \frac{1}{2\bar{C} k^2} =: \frac{C_0}{k^2}$ where $C_0 = \frac{1}{2\bar{C}}$ is a constant independent of $k$.
For this choice of $\varepsilon$, we have $|\hat{v}(\hat{x})|=|\hat{w}(\hat{x})|\ge \frac{1}{2}\|\hat{v}\|_{L^\infty(\hat{K})}$ for all $\hat{x}\in \hat{B}_\varepsilon(\bar{x})\cap\hat{K}$.
Then, we have that 
\begin{align*}
\|\hat{v}\|^r_{L^r(\hat{K})} = \int_{\hat{K}} |\hat{v}(\hat{x})|^r \dd \hat{x} \ge \int_{\hat{B}_{\varepsilon}(\bar{x})} |\hat{v}(\hat{x})|^r \dd \hat{x} \ge \tfrac{1}{2^r} \|\hat{v}\|_{L^\infty(\hat{K})}^r \cdot \text{Vol}(\hat{B}_{\varepsilon}(\bar{x})\cap \hat{K})
  \end{align*}
On the fixed reference triangle/square, the intersection $\hat{B}_{\varepsilon}(\bar{x})\cap \hat{K}$ has its measure bounded below by
\begin{align*}
  \text{Vol}(\hat{B}_{\varepsilon}(\bar{x})\cap \hat{K}) \ge c\varepsilon^2
\end{align*}
for a constant $c>0$ independent of $k$ and $\varepsilon$, and therefore, substituting $\varepsilon$, we find that
\begin{equation*}
  \|\hat{v}\|^r_{L^r(\hat{K})} \ge \tfrac{1}{2^r} \|\hat{v}\|_{L^\infty(\hat{K})}^r \cdot {c}\varepsilon^2= \tfrac{1}{2^r} \|\hat{v}\|_{L^\infty(\hat{K})}^r \cdot \tfrac{{c}\,C_0^2}{k^4}.
\end{equation*}
Rearranging and taking the $r$-th root of both sides, we find that 
\begin{equation*}
  \|v\|_{L^\infty({K})} = \|\hat{v}\|_{L^\infty(\hat{K})} \le 2\left( \tfrac{1}{{c} C^2_0} \right)^{\tfrac{1}{r}} (k^2)^{\tfrac{2}{r}} \|\hat{v}\|_{L^r(\hat{K})} \le C \left(\tfrac{k^2}{h_K}\right)^{\tfrac{2}{r}} \|v\|_{L^r({K})}
\end{equation*}
with $C = 4\left(\frac{{1}}{{c} C^2_0}\right)^{\frac{1}{r}}$.
\end{proof}
    \begin{theorem}[Discrete Sobolev inequality for high-order FEM]\label{thm:discrete-sobolev}
      Let $V_h^k\subset H^1(\Omega)$ be a conforming FEM space of degree $k \ge 1$ on a triangular or quadrilateral mesh $\mathcal T_h$ satisfying the usual mesh assumptions. Then, for all $u_h\in V_h^k$, we have
      \begin{equation*}
      \|u_h\|^2_{L^\infty(\Omega)} \;\le\; C\,
      \left(1+\log \left(\tfrac{Hk^2}{h}\right)\right)\|u_h\|^2_{H^1(\Omega)},
      \end{equation*}
      with $C>0$ independent of $h,H$, and $k$. 
      \end{theorem}
    \begin{proof}

Let $K\in\mathcal T_h$ and $v_h\in V_h^k(K)$. Without loss of generality, we assume $H=1$. Lemma~\ref{lem:infty-Lr-simplex} implies
\[
\|v_h\|_{L^\infty(K)}
\;\le\; C\left(\tfrac{k^2}{h_K}\right)^{2/r}
\|v_h\|_{L^r(K)}.
\]
For each $r\ge1$, the standard Sobolev inequality in 2D yields
\begin{equation}\label{eq:sobolev-H1}
\|v\|_{L^r(K)} \;\le\; C\,r^{1/2}\,\|v\|_{H^1(K)}
\qquad\forall v\in H^1(K),
\end{equation}
with $C$ depending only on the shape of $K$. Taking the supremum over $K\in\mathcal{T}_h$, applying Lemma~\ref{lem:infty-Lr-simplex} and \eqref{eq:sobolev-H1} on each element, and using $(k^2/h_K)^{2/r}\le C(k^2/h)^{2/r}$ together with $\max_K a_K\le \bigl(\sum_K a_K^2\bigr)^{1/2}$, we obtain the global estimate
\begin{equation}\label{eq:global-Linf-r-H1}
\|u_h\|_{L^\infty(\Omega)}
\;\le\; C\,r^{1/2}\left(\tfrac{k^2}{h}\right)^{\tfrac{2}{r}}\left(\sum_{K\in\mathcal T_h}\|u_h\|_{H^1(K)}^2\right)^{\tfrac{1}{2}}
\;=\; C\,r^{1/2}\left(\tfrac{k^2}{h}\right)^{\tfrac{2}{r}}\|u_h\|_{H^1(\Omega)}.
\end{equation}
Now, using $r:=(1+\log(k^2/h))$ we get
\begin{equation}\label{eq:discr-sobolev-unit}
  \begin{split}
  \|u_h\|^2_{L^\infty(\Omega)} \!&\overset{{\eqref{eq:global-Linf-r-H1}}}{\le} \!C\left(1+\log\left(\tfrac{k^2}{h}\right)\right)\left(\tfrac{k^2}{h}\right)^{4/(1+\log(k^2/h))}\|u_h\|^2_{H^1(\Omega)} \\
  &\,\,\,\hspace{0.4mm}{\le} \,\,\hspace{0.5mm}C\left(1+\log\left(\tfrac{k^2}{h}\right)\right)\,\|u_h\|^2_{H^1(\Omega)}
  \end{split}
\end{equation}
as one can show that $\left(\frac{k^2}{h}\right)^{4/(1+\log(k^2/h))}\le e^4.$
\end{proof}

\begin{remark}
  If $\widehat\Omega\subset\mathbb R^2$ is a polygonal domain with
  $\operatorname{diam}(\widehat\Omega)=H>0$ and mesh size $\hat h$,
  let $F:\Omega\to\widehat\Omega$ be an affine map with scaling factor $H$
  and set $\Omega:=H^{-1}\widehat\Omega$, $h:=\hat h/H$.
  For $\hat u_h\in V_h^k(\widehat\Omega)$ define
  $u_h(x) := \hat u_h(F(x))$ on $\Omega$.
  Then there exist $c_1, c_2 > 0$ independent of $H$ such that
  \[
  c_1 \|\hat u_h\|_{H^1(\widehat\Omega)} \le \|u_h\|_{H^1(\Omega)} \le c_2 \|\hat u_h\|_{H^1(\widehat\Omega)},
  \qquad
  \|u_h\|_{L^\infty(\Omega)} = \|\hat u_h\|_{L^\infty(\widehat\Omega)},
  \]
  and the mesh on $\Omega$ has diameter $1$ and size $h=\hat h/H$.
  Applying the unit-diameter estimate \eqref{eq:discr-sobolev-unit} to $u_h$ on $\Omega$
  yields the general form with $\log(Hk^2/\hat h)$.
  \end{remark}

  \subsection{High-order interpolation inequality} 
  To complete the GDSW proof, we need a family of interpolation operators from $V_h^{2k}$ to $V_h^k$, that is uniformly $H^1$-stable independent of $k$, and produces $H^1$-functions on the entire domain. It will be used to interpolate the restriction products $w_i := \theta_i w$ to $V_h^k$. Fortunately, such a family can be constructed using \cite[Theorem 4.4]{eibner_melenk_local_2004}.

    \begin{theorem}[Conforming degree-reduction operators]\label{hointerpolation}
      Let $\Omega \subset \mathbb{R}^2$ be a polygonal domain, and let $\{\mathcal T_h\}_{h>0}$ be a shape-regular family of conforming triangulations or quadrilateral meshes of $\Omega$.
      There exists a family of interpolation operators $I_{k}^h : V_h^{2k} \to V_h^k$ with the following properties.
      
      \begin{itemize}
      \item $I^k_h v_h = v_h$ for all $v_h \in V_h^k.$
      \item For every element $K \in \mathcal T_h$ and every edge $e \subset \partial K$,
      \[
      (I^k_h v_h)|_e = i_{k,e}(v_h|_e),
      \]
      where $i_{k,e}$ denotes the degree $k$ Gauss-Lobatto interpolation onto $\mathbb P_k(e)$.
      \item There exists a constant $C>0$, independent of $h$ and $k$, such that
      \[
      |I^k_h v_h|_{H^1(\Omega)} \le C |v_h|_{H^1(\Omega)}
      \qquad \text{for all } v_h \in V_h^{2k}.
      \]
      \end{itemize}
      
      \end{theorem}

      \begin{proof}
          The family and its properties can be constructed as a special case of \cite[Theorem 4.4]{eibner_melenk_local_2004} by choosing $p = 2k$ and $p' = k$ and bounding the $p$-dependent constants by 3. Further, we are using conforming finite element spaces, therefore the interpolated functions are continuous and therefore in $H^1$.
      \end{proof}

      \begin{remark}\label{sametrace}
        We note, that for a fixed $k$ and possibly distinct functions $u_h\in V^{k+1}_h$ and $v_h\in V^{2k}_h$ that have the same trace, so $u_h|_{\partial\Omega} = v_h|_{\partial\Omega}$, the respectively interpolated functions $I_{k}^h u_h, I_{k}^h v_h \in V^k_h$ have the same trace as well, by construction.
      \end{remark}
      As stated earlier, this will play the role of \cite[Lemma 4.31]{toselli_widlund_ddm_2005} in the edge lemma proof, and in the GDSW proof.

\subsection{Edge lemma}
The edge lemma is a key ingredient in the proof of the condition number bound for GDSW and RGDSW. It provides bounds on the energy of the discrete harmonic extension of the product of an edge function and an arbitrary function, as well as bounds on the energy and $L^2$-norm of the edge function itself. For $k=1$ it was shown in \cite[Lemma 4.4]{klawonn_analysis_2008} for John domains, and here we extend it to arbitrary polynomial degrees. We first clarify the construction of high-order vertex and edge functions.
\begin{definition}[High-order vertex/edge function]\label{def:vertex-edge-fn}
  Let $\Omega_i$ be a John domain and $\Gamma \subset \partial \Omega_i$ be either a vertex or an open edge. We say that $\theta_{\Gamma} \in V^{k}_h(\Omega_i)$ is a high-order vertex or edge function respectively, if $\theta_{\Gamma}$ is a finite element function which equals $1$ at all DOFs of $\Gamma$, vanishes at the other nodes on $\partial \Omega_i$, is interpolated piecewise linearly on all the boundary degrees of freedom and is discrete harmonic in $\Omega_i$.
\end{definition}
Since the decomposition into vertex and edge functions needs to be a partition of unity, this construction of vertex and edge functions with piecewise linear interpolation on the boundary values is sufficient for our case and allows us to use large parts of the proof of the edge lemma for $k=1$ to prove the high-order case.
\begin{definition}[Discrete harmonic extension]\label{def:discrete-harmonic-extension}
  Let $\Omega_i$ be a subdomain and let $g_h$ be the trace of a finite element function in $V_h^k(\Omega_i)$ on $\partial\Omega_i$. The discrete harmonic extension of $g_h$ into $\Omega_i$ is the function $\mathcal{H}_i(g_h) \in V_h^k(\Omega_i)$ satisfying
  \begin{equation*}
    \mathcal{H}_i(g_h)|_{\partial\Omega_i}=g_h,
    \qquad
    a_i(\mathcal{H}_i(g_h),v_h)=0 \quad \forall v_h\in V_{h,0}^k(\Omega_i),
  \end{equation*}
  where $V_{h,0}^k(\Omega_i):=\{v_h\in V_h^k(\Omega_i): v_h=0 \text{ on } \partial\Omega_i\}$ and $a_i(w,v):=\int_{\Omega_i}\nabla w\cdot\nabla v\,\dd x$. Equivalently, $\mathcal{H}_i(g_h)$ is the energy-minimizing extension of $g_h$ in $V_h^k(\Omega_i)$.
\end{definition}

\begin{theorem}[High-order edge lemma]\label{thm:edgelemma}
  Let $\Omega_i$ be a John domain, $\mathcal{E}_{ij} \subset \partial \Omega_i$ be an edge, and 
$\theta_{\mathcal{E}_{ij}} \in V^{k}_h(\Omega_i)$ be a high-order edge function. Then we have the following bounds:
\begin{align*}
  |\mathcal{H}_i(I^k_h(\theta_{\mathcal{E}_{ij}}u))|^2_{H^1(\Omega_i)} 
  &\le C \, \left(1 + \log\left(\tfrac{H_ik^2}{h_i}\right)\right)^2 \, \|u\|^2_{H^1(\Omega_i)}, \\
  |\theta_{\mathcal{E}_{ij}}|^2_{H^{1}(\Omega_i)} &\le C \left(1 + \log\left(\tfrac{H_i}{h_i}\right)\right), \\
  \|\theta_{\mathcal{E}_{ij}}\|^2_{L^{2}(\Omega_i)} &\le C H_i^2 \left(1 + \log\left(\tfrac{H_i}{h_i}\right)\right).
\end{align*}

\end{theorem}
\begin{proof}
  We proceed similarly to \cite[Lemma 4.4]{klawonn_analysis_2008} and show the case of triangular meshes first. The construction of the piecewise $\mathbb{P}_1$-function $\vartheta_{\mathcal{E}_{ij}}$ is independent of the considered ansatz space, therefore also independent of $k$ and can be done the same way as in the cited work. We summarize the construction briefly and refer to \cite[Lemma 4.4]{klawonn_analysis_2008} for details. 
  Using that $\Omega$ is a John domain, a finite series of rectangles $R_{l,l+1}$ is constructed which covers the John curves and is a subset of $\Omega$. It furthermore splits $\Omega$ into three parts. On the part connected to $\mathcal{E}_{ij}$, the function $\vartheta_{\mathcal{E}_{ij}}$ is set to one, the part covered by the rectangles is triangulated and the function decreases piecewise linearly from 1 to 0 on the triangulation, and on the remaining part the function is set to zero.  For an arbitrary degree $k$, we pick the corresponding operator $I^k_h$ of the family of interpolation operators in \cref{hointerpolation}, and consider the functions $\vartheta_{\mathcal{E}_{ij}}u_h \in V^{k+1}_h$ and $\theta_{\mathcal{E}_{ij}}u_h \in V^{2k}_h$. Since these functions have the same trace, the interpolated functions $I^k_h(\theta_{\mathcal{E}_{ij}}u_h) \in V^{k}_h, I^k_h(\vartheta_{\mathcal{E}_{ij}}u_h) \in V^{k}_h$ have the same trace as well (see \cref{sametrace}). Therefore, by energy minimality of the discrete harmonic extension in the space $V^k_h$, we have that  
  \begin{equation*}
    |\mathcal{H}_i(I^k_h(\theta_{\mathcal{E}_{ij}}u))|_{H^1(\Omega_i)} \le |I^k_h(\vartheta_{\mathcal{E}_{ij}}u)|_{H^1(\Omega_i)}
  \end{equation*}
  Further, by the uniform $H^1$-stability of the $I^k_h$ operator
  \begin{align*}
    |I^k_h(\vartheta_{\mathcal{E}_{ij}}u)|_{H^1(\Omega_i)}
    &\le C |\vartheta_{\mathcal{E}_{ij}}u|_{H^1(\Omega_i)} && \hspace{2.45cm}\makebox[0pt][r]{$[\text{\cref{hointerpolation}}]$}\\
    &\le C (\|u\nabla\vartheta_{\mathcal{E}_{ij}}\|_{L^2(\Omega_i)} + \|\vartheta_{\mathcal{E}_{ij}}\nabla u\|_{L^2(\Omega_i)}) && \hspace{2.45cm}\makebox[0pt][r]{$[\text{product rule, $\triangle$-ineq.}]$}\\
    &\le C (\| u \|_{L^\infty(\Omega_i)} |\vartheta_{\mathcal{E}_{ij}}|_{H^1(\Omega_i)} +\| \vartheta_{\mathcal{E}_{ij}} \|_{L^\infty(\Omega_i)} |u|_{H^1(\Omega_i)}) &&\\
    &\le C(\| u \|_{L^\infty(\Omega_i)} |\vartheta_{\mathcal{E}_{ij}}|_{H^1(\Omega_i)} + |u|_{H^1(\Omega_i)}) && \hspace{2.45cm}\makebox[0pt][r]{$[\|\vartheta_{\mathcal{E}_{ij}}\|_{L^\infty(\Omega_i)}\le 1]$}\\
    &\le  C\left(\left(1+\log \left(\tfrac{H_i k^2}{h_i}\right)\right)^{1/2}|\vartheta_{\mathcal{E}_{ij}}|_{H^1(\Omega_i)} + 1\right) \,\|u\|_{H^1(\Omega_i)} && \\
    &\le C\left(1+\log \left(\tfrac{H_ik^2}{h_i}\right)\right) \,\|u\|_{H^1(\Omega_i)}. &&
  \end{align*}
  Here we used \cref{thm:discrete-sobolev}, the triangle inequality, and the fact that $0\le \vartheta_{\mathcal{E}_{ij}}\le 1$ by a maximum principle for discrete harmonic piecewise linear functions. The last inequality follows from the fact that $\vartheta_{\mathcal{E}_{ij}}$ is a piecewise linear function and therefore $|\vartheta_{\mathcal{E}_{ij}}|^2_{H^1(\Omega_i)}\le C\Bigl(1+\log \left(\frac{H_i}{h_i}\right)\Bigr)$ holds from the proof of the edge lemma in \cite[Lemma 4.4]{klawonn_analysis_2008}.

  For the bound on the $H^1$-seminorm, let $\theta^{[k]}_{\mathcal{E}_{ij}}$ denote the discrete harmonic extension in $V^k_h$ of the piecewise linearly interpolated boundary values of $\theta_{\mathcal{E}_{ij}}|_{\partial\Omega_i}$ in the element edge degrees of freedom. Since the boundary values are the same, and $V^1_h\subset V^k_h$, we can estimate $|\theta^{[k]}_{\mathcal{E}_{ij}}|_{H^1(\Omega_i)}\le |\theta^{[1]}_{\mathcal{E}_{ij}}|_{H^1(\Omega_i)}$ due to energy minimality, and perform the rest of the estimate as in the $k=1$ case.
  
  For the $L^2(\Omega_i)$-bound, we use Poincaré's inequality and the above inequality. Since $\theta_{\mathcal{E}_{ij}}-\vartheta_{\mathcal{E}_{ij}}$ is zero on $\partial \Omega_i$, we have that 
  \begin{align*}
    \|\theta_{\mathcal{E}_{ij}}\|_{L^2(\Omega_i)} &\le \|\vartheta_{\mathcal{E}_{ij}}\|_{L^2(\Omega_i)} + \|\theta_{\mathcal{E}_{ij}}-\vartheta_{\mathcal{E}_{ij}}\|_{L^2(\Omega_i)} \\
    &\le |\Omega_i|^{1/2} + C_{\text{PF}}|\Omega_i|^{1/2}|\theta_{\mathcal{E}_{ij}}-\vartheta_{\mathcal{E}_{ij}}|_{H^1(\Omega_i)} \\
    &\le C H_i (1+ 2C_{\text{PF}}|\vartheta_{\mathcal{E}_{ij}}|_{H^1(\Omega_i)}) \\
    &\le C H_i \left(1 + \log\left(\tfrac{H_i}{h_i}\right)\right)^{1/2}
    \end{align*}
    where $C_{\text{PF}}$ is the Poincaré constant. Squaring both sides yields the desired result.

  If $V^k_h$ is a $\mathbb{Q}_k$-space, let $\mathcal{Q}_h$ be the quadrilateral mesh inside the union of the rectangles and $\vartheta^{{\mathcal{Q}_h}}_{\mathcal{E}_{ij}}$ be piecewise bilinear on ${\mathcal{Q}_h}$. Let ${\mathcal{T}_h}$ be the mesh constructed by subdividing each quadrilateral $Q\in {\mathcal{Q}_h}$ into two triangles $T_1$ and $T_2$, and $\vartheta^{{\mathcal{T}_h}}_{\mathcal{E}_{ij}}$ the corresponding function on ${\mathcal{T}_h}$. Since $\vartheta^{{\mathcal{Q}_h}}_{\mathcal{E}_{ij}}$ is harmonic on each element $Q\in {\mathcal{Q}_h}$, we have that $|\vartheta^{{\mathcal{Q}_h}}_{\mathcal{E}_{ij}}|_{H^1(Q)}\le |\vartheta^{{\mathcal{T}_h}}_{\mathcal{E}_{ij}}|_{H^1(T_1\cup T_2)}$. The rest of the proof can be carried out similarly to the $\mathbb{P}_k$-case.
\end{proof}

\begin{remark}\label{rem:vertex-fn}
  For high-order vertex functions $\theta_{\mathcal{V}_l}$, similar bounds can be shown by the same arguments, using the fact $|\theta^{[k]}_{\mathcal{V}_l}|_{H^1(\Omega_i)}\le |\theta^{[1]}_{\mathcal{V}_l}|_{H^1(\Omega_i)}$ and a construction of a piecewise linear function $\vartheta_{\mathcal{V}_l}$ that equals one at the vertex $\mathcal{V}_l$ and zero at all other vertices and edges of $\Omega_i$.
\end{remark}

    \subsection{Proof of the condition number bound}
We now give a proof of the high-order GDSW and RGDSW condition number bound, giving special attention to the differences that result from the use of high-order finite elements. 
We first recall that we denote by $\Omega_i$ the nonoverlapping subdomains and by $\Omega'_i$ the overlapping ones, where each one has an overlap of size $\delta_i$. We denote the overlapping strip by $\Omega_{i,\delta_i}:= \Omega'_i\setminus\Omega_i$ and the mesh size on $\Omega_i$ by $h_i$. We denote the diameter of $\Omega_i$ by $H_i$ and $H:=\max_{1\le i\le N} H_i$. We assume that the subdomains are shape-regular and that the mesh on each subdomain is shape-regular and quasi-uniform.

We apply the abstract two-level overlapping Schwarz theory \cite[Chapter~2]{toselli_widlund_ddm_2005}. Assuming exact local solvers, we define the local and coarse bilinear forms by
\[
  \tilde{a}_i(u_i,v_i) := a(R_i^{T}u_i,R_i^{T}v_i), \quad u_i,v_i \in V_i, \quad i=0,1,\dots,N,
\]
and we set $R_0 := \Phi^T$, where $\Phi$ is the coarse basis matrix introduced in Section~\ref{sec:GDSW}, so that $V_0 = \Phi^T V$.
By the abstract theory, it suffices to construct a stable decomposition: For every $u\in V$, find $u_i\in V_i$, $i=0,\dots,N$, such that
\[
  u = \sum_{i=0}^N R_i^{T} u_i \qquad\text{and}\qquad \sum_{i=0}^N \tilde{a}_i(u_i,u_i) \le C_0^{\,2}\, a(u,u)
\]
for some constant $C_0>0$. The constant $C_0^2$ then bounds the condition number from above, while the lower bound $a(u,u)\le \sum_{i=0}^N\tilde{a}_i(u_i,u_i)$ follows from the partition-of-unity property. Let $\theta_i$ for $i=1,\dots,N$ be the partition-of-unity functions associated with the overlapping decomposition, and let $I^k_h$ be the interpolation operator from \cref{hointerpolation}. 
We define {$u_i := R_i(I^k_h(\theta_i (u-u_0)))\in V_i$ and the coarse interpolation $I_{\mathrm{GDSW}}^H: V \rightarrow V_0$} by 
\begin{equation*}
u_0 := I_{\mathrm{GDSW}}^H u := \sum_l u(\mathcal{V}_l) \theta_{\mathcal{V}_l} + \sum_{jm} 
\bar{u}_{\mathcal{E}_{jm}}\,\theta_{\mathcal{E}_{jm}},
\end{equation*} where $\theta_{\mathcal{V}_l}$, $\theta_{\mathcal{E}_{jm}}$ are the vertex and edge basis functions of Definition~\ref{def:vertex-edge-fn} and $\bar{u}_{\mathcal{E}_{jm}}$ is the mean value of $u$ on the edge $\mathcal{E}_{jm}$.

Since $R_0 = \Phi^T$, the operator $I_{\mathrm{GDSW}}^H$ coincides with the coarse space restriction $R_0$ in the abstract Schwarz framework. Moreover, $I_{\mathrm{GDSW}}^H u_0 = u_0$, since $u_0$ is a linear combination of the high-order vertex and edge basis functions, which satisfy the interpolation conditions of \Cref{def:vertex-edge-fn} by construction.
We take care of the coarse part first. We note that, by an abuse of notation, $u_0$ alternates between coarse functions in $V_0$ and their representation as functions in $V$. We have that
\begin{equation*}
  \tilde{a}_0(u_0,u_0) = a(I_{\mathrm{GDSW}}^H u_0,I_{\mathrm{GDSW}}^H u_0) = a(u_0,u_0) = |u_0|_{H^{1}(\Omega)}^{2} = \sum_{i=1}^{N} |u_0|_{H^{1}(\Omega_i)}^{2}.
\end{equation*}
On each subdomain, we estimate
\begin{equation}\label{eq:u0estimate}
\begin{split}
  |u_0|_{H^{1}(\Omega_i)}^{2} &= |u_0- \bar{u}_{\Omega_i}|_{H^{1}(\Omega_i)}^{2} \\
  &= \left\lvert \sum_{l} (u(\mathcal{V}_{l})-\bar{u}_{\Omega_i})\,\theta_{\mathcal{V}_{l}} + \sum_{jm} (\bar{u}_{\mathcal{E}_{jm}}-\bar{u}_{\Omega_i})\,\theta_{\mathcal{E}_{jm}} \right\rvert_{H^{1}(\Omega_i)}^{2} \hspace{1.65cm} [\text{def.\ }I_{\mathrm{GDSW}}^H]\\
  &\le \sum_l ||u(\mathcal{V}_{l})-\bar{u}_{\Omega_i}||^2_{L^\infty(\Omega_i)} |\theta_{\mathcal{V}_{l}}|^2_{H^{1}(\Omega_i)} + \sum_{jm} ||\bar{u}_{\mathcal{E}_{jm}}-\bar{u}_{\Omega_i}||^2_{L^\infty(\Omega_i)} |\theta_{\mathcal{E}_{jm}}|^2_{H^{1}(\Omega_i)} \\
  &\le  ||u-\bar{u}_{\Omega_i}||^2_{L^\infty(\Omega_i)} \left(C \sum_l |\theta_{\mathcal{V}_l}|_{H^{1}(\Omega_i)}^{2} + C \sum_{jm} |\theta_{\mathcal{E}_{jm}}|_{H^{1}(\Omega_i)}^{2}\right) \\
  &\le C \left(1+\log\left(\tfrac{H_i}{h_i}\right)\right)\left(1+\log\left(\tfrac{H_ik^2}{h_i}\right)\right)|u|^2_{H^1(\Omega_i)}. \!\qquad[\text{\cref{thm:discrete-sobolev},\ \text{\cref{thm:edgelemma}}}]
  \raisetag{-4.9\baselineskip}
\end{split}
\end{equation}
Now, for $w:= u-u_0$, it follows that
\begin{equation}\label{wH1}
  |w|_{H^{1}(\Omega_i)}^{2} = |u-u_0|_{H^{1}(\Omega_i)}^{2} \le C \left(1+\log\left(\tfrac{H_i}{h_i}\right)\right)\left(1+\log\left(\tfrac{H_ik^2}{h_i}\right)\right)|u|^2_{H^1(\Omega_i)}.
\end{equation}
For a constant $u_c$ one can show that $I_{\mathrm{GDSW}}^H u_c = u_c$. Let therefore $\bar{u}_{\Omega_i} = \frac{1}{|\Omega_i|}\int_{\Omega_i} u \dd x$ be the mean value of $u$ on $\Omega_i$ and define $\hat{u} = u - \bar{u}_{\Omega_i}$. Using similar arguments as above we find that
\begin{equation}\label{eq:IHuhat-estimate}
\begin{split}
  \|I_{\mathrm{GDSW}}^H \hat{u}&\|^2_{L^2(\Omega_i)} \!= \left\lvert\left\lvert \sum_{l} (u(\mathcal{V}_{l})-\bar{u}_{\Omega_i})\,\theta_{\mathcal{V}_{l}} + \sum_{jm} (\bar{u}_{\mathcal{E}_{jm}}-\bar{u}_{\Omega_i})\,\theta_{\mathcal{E}_{jm}} \right\rvert\right\rvert_{L^{2}(\Omega_i)}^{2} \\
  &\le \sum_l \|u(\mathcal{V}_{l})-\bar{u}_{\Omega_i}\|^2_{L^\infty(\Omega_i)} \|\theta_{\mathcal{V}_{l}}\|_{L^{2}(\Omega_i)}^2 + \sum_{jm} \|\bar{u}_{\mathcal{E}_{jm}}-\bar{u}_{\Omega_i}\|^2_{L^\infty(\Omega_i)} \|\theta_{\mathcal{E}_{jm}}\|_{L^{2}(\Omega_i)}^2 \\
  &\le  \|u-\bar{u}_{\Omega_i}\|^2_{L^\infty(\Omega_i)} \left(C \sum_l \|\theta_{\mathcal{V}_l}\|_{L^{2}(\Omega_i)}^{2} + C \sum_{jm} \|\theta_{\mathcal{E}_{jm}}\|_{L^{2}(\Omega_i)}^{2}\right)\\
  &\le C H_i^2\left(1+\log\left(\tfrac{H_i}{h_i}\right)\right)\left(1+\log\left(\tfrac{H_ik^2}{h_i}\right)\right)|u|^2_{H^{1}(\Omega_i)}.
\raisetag{-4.8\baselineskip}
\end{split}
\end{equation}
For the local contributions we require the following estimate from the literature.
\begin{lemma}\label{lem:overlap-strip}\cite[Lemma 3.10]{toselli_widlund_ddm_2005}
  There exists a constant $C$ such that
 \begin{equation*}
  \| u \|^2_{L^2(\Omega_{i,\delta_i})} \le C \delta_i^2\left(\left(1+\tfrac{H_i}{\delta_i}\right)|u|^2_{H^1(\Omega'_i)} + \tfrac{1}{H_i\delta_i}\|u\|^2_{L^2(\Omega'_i)}\right), \qquad u\in H^1(\Omega'_i)
 \end{equation*}
\end{lemma}
Using this lemma, we can estimate the local contributions as follows:
\begin{align}
  \tilde{a}_i(u_i,u_i) &= |I^k_h(\theta_i w)|^2_{H^1(\Omega'_i)} && \notag\\
  &\le C | \theta_i w |^2_{H^1(\Omega'_i)} && \hspace{2.5cm}\makebox[0pt][r]{$\qquad[\text{\cref{hointerpolation}}]$}\notag\\
  &\le C \|\theta_i \nabla w\|_{L^2(\Omega'_i)}^2 + C \|w \nabla \theta_i \|_{L^2(\Omega'_i)}^2 && \notag\\
  &\le C\|\nabla w\|_{L^2(\Omega'_i)}^2  + \tfrac{C}{\delta_i^2} \|w \|_{L^2(\Omega_{i,\delta_i})}^2 && \hspace{2.5cm}\makebox[0pt][r]{$\qquad[\theta_i\le 1,\;|\nabla\theta_i|\le C/\delta_i]$} \label{localContr}\\
  &\le C|w|_{H^1(\Omega'_i)}^2 + C\left(\left(1+\tfrac{H_i}{\delta_i}\right)|w|^2_{H^1(\Omega'_i)} + \tfrac{1}{H_i\delta_i}\|w\|^2_{L^2(\Omega'_i)}\right) && \hspace{2.5cm}\makebox[0pt][r]{$\qquad[\text{\cref{lem:overlap-strip}}]$}\notag\\
  &\le C\left(\left(1+\tfrac{H_i}{\delta_i}\right)|w|^2_{H^1(\Omega'_i)} + \tfrac{1}{H_i\delta_i}\|w\|^2_{L^2(\Omega'_i)}\right) \notag
\end{align}
Here, we used that $0\le \theta_i\le 1$, as it is a linear or bilinear {(therefore harmonic)} function with boundary values between 0 and 1, and $\theta_i \equiv 1$ on $\Omega_i$ therefore $\nabla \theta_i \equiv 0$ on $\Omega_i$. Further, on the overlap strip $\Omega_{i,\delta_i}$ the function $\nabla \theta_i$ is bounded by $C/\delta_i$ where $C$ is the mesh regularity constant. This can be derived by mapping to the respective reference element and using that the differential of the mapping is bounded by the shape regularity constant.
Using the above definition of $\hat{u}$, a Poincaré inequality, and \cref{eq:IHuhat-estimate} we find that
\begin{align}
  \|w\|^2_{L^2(\Omega_i)} &= \|u-I_{\mathrm{GDSW}}^Hu\|^2_{L^2(\Omega_i)} && \notag\\
  &= \|u-\bar{u}_{\Omega_i} -  (I_{\mathrm{GDSW}}^H u - \bar{u}_{\Omega_i})\|^2_{L^2(\Omega_i)} && \hspace{3.45cm}\makebox[0pt][r]{$\qquad[\pm\,\bar{u}_{\Omega_i}]$}\notag\\
  &= \|u-\bar{u}_{\Omega_i} - (I_{\mathrm{GDSW}}^H u - I_{\mathrm{GDSW}}^H(\bar{u}_{\Omega_i}))\|^2_{L^2(\Omega_i)} && \hspace{3.45cm}\makebox[0pt][r]{$\qquad[I_{\mathrm{GDSW}}^H u_c = u_c]$}\notag\\
  &= \|\hat{u} - I_{\mathrm{GDSW}}^H \hat{u}\|^2_{L^2(\Omega_i)} && \hspace{3.45cm}\makebox[0pt][r]{$\qquad[\text{def.\ }\hat{u}]$} \label{wL2}\\
  &\le 2( \|\hat{u}\|^2_{L^2(\Omega_i)} + \|I_{\mathrm{GDSW}}^H \hat{u}\|^2_{L^2(\Omega_i)}) && \notag\\
  &\le C H_i^2\left(1+\log\left(\tfrac{H_i}{h_i}\right)\right)\left(1+\log\left(\tfrac{H_ik^2}{h_i}\right)\right)|u|^2_{H^1(\Omega_i)}. && \hspace{3.45cm}\makebox[0pt][r]{$\qquad[\text{\cref{lem:poincare}},\ \rm eq.\, \eqref{eq:IHuhat-estimate}]$}\notag
\end{align}
 With \eqref{localContr}, we have
\begin{align*}
  \tilde{a}_i(u_i,u_i)
  &\le C\left(\left(1+\tfrac{H_i}{\delta_i}\right)|w|^2_{H^1(\Omega'_i)} + \tfrac{1}{H_i\delta_i}\|w\|^2_{L^2(\Omega'_i)}\right) \\
  &\le C \left(1+\tfrac{H_i}{\delta_i}\right) \condsum{j}{\Omega'_j \cap \Omega'_i \neq \emptyset} |w|^2_{H^1(\Omega_j)} + C\tfrac{1}{H_i\delta_i}\condsum{j}{\Omega'_j \cap \Omega'_i \neq \emptyset} \|w\|^2_{L^2(\Omega_j)}.
\end{align*}
Using the assumption that there is a constant such that $C H\le H_i$ for every $1\le i\le N$, \cref{wH1}
and \cref{wL2}, we have 
\begin{align*}
  \tilde{a}_i(u_i,u_i) &\le C \left(1+\tfrac{H_i}{\delta_i}\right) \condsum{j}{\Omega'_j \cap \Omega'_i \neq \emptyset} \left(1+\log\left(\tfrac{H_j}{h_j}\right)\right)\left(1+\log\left(\tfrac{H_jk^2}{h_j}\right)\right) |{u}|^2_{H^1(\Omega_j)}. 
\end{align*}
Now, the result for GDSW follows from summing over all subdomains: 
\begin{align*}
  \sum_{i=1}^{N} \tilde{a}_i(u_i,u_i) & \le \sum_{i=1}^{N}C\left(1+\tfrac{H_i}{\delta_i}\right) \condsum{j}{\Omega'_j \cap \Omega'_i \neq \emptyset} \left(1+\log\left(\tfrac{H_j}{h_j}\right)\right)\left(1+\log\left(\tfrac{H_jk^2}{h_j}\right)\right)|u|_{H^1(\Omega_j)}^2 \\
  &\le C \bigl(1+\tfrac{H}{\delta}\bigr)\left(1+\log\left(\tfrac{H}{h}\right)\right)\bigl(1+\log\bigl(\tfrac{Hk^2}{h}\bigr)\bigr)a(u,u).
\end{align*}
  For RGDSW, we can write each basis function $\phi^{\mathcal{P}_{\mathbf{v}}}$ as
  \[
     \phi^{\mathcal{P}_{\mathbf{v}}} = \theta_{\mathcal{V}_\mathbf{v}} + \frac{1}{n_{\mathcal{E}_\mathbf{v}}} \sum_{\mathcal{E}\in \mathcal{E}_\mathbf{v}} \theta_{\mathcal{E}},
  \]
  where $\theta_{\mathcal{V}_\mathbf{v}}$ and $\theta_{\mathcal{E}}$ are the standard GDSW vertex and edge basis functions of \Cref{def:vertex-edge-fn}. By the triangle inequality, the bounds of \Cref{thm:edgelemma} and \Cref{rem:vertex-fn} yield $|\phi^{\mathcal{P}_{\mathbf{v}}}|^2_{H^1(\Omega_i)} \le C\bigl(1+\log(H_i/h_i)\bigr)$, with $C$ independent of $k$. Define the RGDSW coarse interpolation $I_{\mathrm{RGDSW}}^H u := \sum_{\mathbf{v}} u(\mathbf{v})\,\phi^{\mathcal{P}_{\mathbf{v}}}$. The analog of \eqref{eq:u0estimate} reads
  \begin{equation*}
  \begin{split}
    |I_{\mathrm{RGDSW}}^H u|^2_{H^1(\Omega_i)}
    &= |I_{\mathrm{RGDSW}}^H u - \bar{u}_{\Omega_i}|^2_{H^1(\Omega_i)} \\
    &= \left\lvert \sum_{\mathbf{v}} \bigl(u(\mathbf{v}) - \bar{u}_{\Omega_i}\bigr)\, \phi^{\mathcal{P}_{\mathbf{v}}} \right\rvert^2_{H^1(\Omega_i)} \\
    &\le \|u - \bar{u}_{\Omega_i}\|^2_{L^\infty(\Omega_i)} \sum_{\mathbf{v}} | \phi^{\mathcal{P}_{\mathbf{v}}}|^2_{H^1(\Omega_i)} \\
    &\le C\left(1+\log\left(\tfrac{H_i}{h_i}\right)\right)\left(1+\log\left(\tfrac{H_ik^2}{h_i}\right)\right)|u|^2_{H^1(\Omega_i)},
  \end{split}
  \end{equation*}
  where the last inequality uses \Cref{thm:discrete-sobolev} for the $L^\infty$-factor, the bound $|\phi^{\mathcal{P}_{\mathbf{v}}}|^2_{H^1(\Omega_i)}\le C(1+\log(H_i/h_i))$ derived above, and the fact that each subdomain has a finite number of vertices in the non-overlapping decomposition. The rest of the proof of \Cref{thm:condbound} can be carried out analogously to the one for GDSW.\hfill\proofbox

\begin{remark}
  With the current available theoretical tools, we were unable to derive a similar result for the virtual element method. As far as we are aware, there is no suitable family of interpolation operators that map between the VEM spaces $V^{2k}_h$ and $V^k_h$ that are uniformly $H^1$-stable and have optimal approximation properties. Furthermore, in order to prove a discrete Sobolev inequality for degree $k$ VEM functions, one needs to control the nonpolynomial part of the VEM functions, which is, to our knowledge only possible with the very general result in \cite[(4.5.3) Lemma]{brenner_mathematical_2008} for every fixed $k$, which then yields the correct bound, but the constant does depend on $k$ as well. Finally, the high-order edge lemma also relies on the above mentioned interpolation operators, and therefore the same issues arise.
\end{remark}

\begin{figure}[t]
\centering
\includegraphics[width=\textwidth]{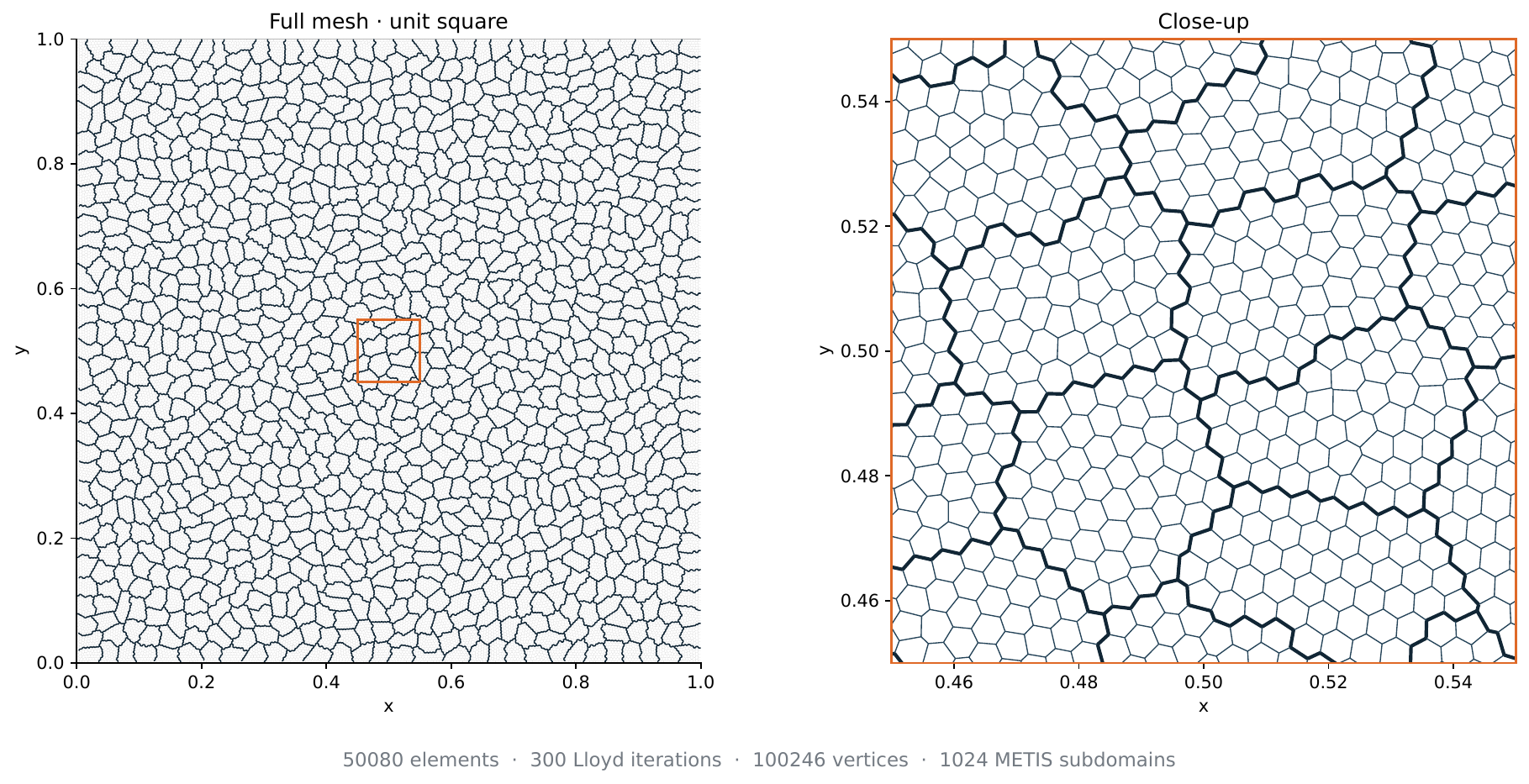}
\caption{Lloyd-regularized polygonal mesh with 50\,080 elements obtained with 300 Lloyd iterations, partitioned into 1\,024 subdomains by METIS.}\label{fig:2d_lloyd_mesh}
\end{figure}

\begin{table}[t]
\centering
\caption{Condition numbers and PCG iteration counts for $\mathbb{P}_k$-FEM (orthogonal polynomial basis) and VEM (d-recipe stabilization, orthogonal polynomial basis) on a structured triangular mesh obtained from a $160\times 160$ square mesh by fixed-diagonal subdivision (overlap 1).}\label{tab:2d_160_kappa_tri}
\begin{adjustbox}{width=\textwidth}
\begin{tabular}{|c|c|r|r||r|r||r|r||r|r||r|r||r|r|}
\cline{3-14}
\multicolumn{2}{c|}{} & nSub & Coarse & \multicolumn{2}{c||}{$k=1$} & \multicolumn{2}{c||}{$k=3$} & \multicolumn{2}{c||}{$k=5$} & \multicolumn{2}{c||}{$k=7$} & \multicolumn{2}{c|}{$k=9$} \\
\cline{3-14}
\multicolumn{2}{c|}{} &  &  & $\kappa$ & it & $\kappa$ & it & $\kappa$ & it & $\kappa$ & it & $\kappa$ & it \\
\cline{3-14}
\noalign{\vskip 0.6ex}
\cline{1-14}
\multirow{8}{*}{\STAB{\rotatebox[origin=c]{90}{FEM}}}
& \multirow{4}{*}{\STAB{\rotatebox[origin=c]{90}{GDSW}}}
& $4\times 4$ & 33 & 26.01 & 37 & 39.03 & 46 & 43.20 & 49 & 45.19 & 52 & 46.32 & 54 \\
& & $8\times 8$ & 161 & 17.44 & 35 & 25.26 & 44 & 28.14 & 48 & 29.89 & 49 & 31.16 & 50 \\
& & $16\times 16$ & 705 & 12.53 & 32 & 18.59 & 41 & 20.98 & 44 & 22.56 & 45 & 23.75 & 48 \\
& & $32\times 32$ & 2\,945 & 8.93 & 28 & 14.36 & 37 & 16.55 & 41 & 18.04 & 44 & 19.20 & 46 \\
\cline{2-14}
& \multirow{4}{*}{\STAB{\rotatebox[origin=c]{90}{RGDSW}}}
& $4\times 4$ & 9 & 24.43 & 37 & 36.58 & 47 & 39.91 & 50 & 41.57 & 53 & 42.61 & 54 \\
& & $8\times 8$ & 49 & 17.24 & 37 & 24.77 & 46 & 27.36 & 49 & 28.91 & 51 & 30.05 & 54 \\
& & $16\times 16$ & 225 & 12.54 & 33 & 18.56 & 43 & 20.89 & 46 & 22.42 & 48 & 23.57 & 49 \\
& & $32\times 32$ & 961 & 8.86 & 28 & 14.31 & 37 & 16.52 & 41 & 18.03 & 44 & 19.19 & 45 \\
\hline
\hline
\multirow{8}{*}{\STAB{\rotatebox[origin=c]{90}{VEM}}}
& \multirow{4}{*}{\STAB{\rotatebox[origin=c]{90}{GDSW}}}
& $4\times 4$ & 33 & 26.01 & 37 & 46.51 & 64 & 48.22 & 64 & 49.24 & 67 & 50.65 & 69 \\
& & $8\times 8$ & 161 & 17.44 & 35 & 28.16 & 56 & 31.65 & 59 & 33.90 & 59 & 35.59 & 59 \\
& & $16\times 16$ & 705 & 12.53 & 32 & 20.34 & 52 & 23.93 & 56 & 26.31 & 57 & 28.10 & 58 \\
& & $32\times 32$ & 2\,945 & 8.93 & 28 & 15.63 & 48 & 19.26 & 52 & 21.67 & 55 & 23.50 & 56 \\
\cline{2-14}
& \multirow{4}{*}{\STAB{\rotatebox[origin=c]{90}{RGDSW}}}
& $4\times 4$ & 9 & 24.43 & 37 & 43.64 & 65 & 45.29 & 61 & 45.99 & 64 & 47.20 & 65 \\
& & $8\times 8$ & 49 & 17.24 & 37 & 27.67 & 60 & 31.01 & 61 & 33.17 & 65 & 34.79 & 66 \\
& & $16\times 16$ & 225 & 12.54 & 33 & 20.32 & 55 & 23.88 & 57 & 26.23 & 60 & 28.01 & 62 \\
& & $32\times 32$ & 961 & 8.86 & 28 & 15.64 & 48 & 19.27 & 53 & 21.69 & 56 & 23.52 & 58 \\
\hline
\end{tabular}
\end{adjustbox}
\end{table}
\begin{table}[h]
\centering
\caption{Condition numbers and PCG iteration counts for $\mathbb{Q}_k$-FEM (orthogonal polynomial basis) and VEM (d-recipe stabilization, orthogonal polynomial basis) on a $160\times 160$ square mesh with overlap 1.}\label{tab:2d_160_kappa}
\begin{adjustbox}{width=\textwidth}
\begin{tabular}{|c|c|r|r||r|r||r|r||r|r||r|r||r|r|}
\cline{3-14}
\multicolumn{2}{c|}{} & nSub & Coarse & \multicolumn{2}{c||}{$k=1$} & \multicolumn{2}{c||}{$k=3$} & \multicolumn{2}{c||}{$k=5$} & \multicolumn{2}{c||}{$k=7$} & \multicolumn{2}{c|}{$k=9$} \\
\cline{3-14}
\multicolumn{2}{c|}{} &  &  & $\kappa$ & it & $\kappa$ & it & $\kappa$ & it & $\kappa$ & it & $\kappa$ & it \\
\cline{3-14}
\noalign{\vskip 0.6ex}
\cline{1-14}
\multirow{8}{*}{\STAB{\rotatebox[origin=c]{90}{FEM}}}
& \multirow{4}{*}{\STAB{\rotatebox[origin=c]{90}{GDSW}}}
& $4\times 4$ & 33 & 25.27 & 34 & 38.92 & 43 & 43.12 & 48 & 45.12 & 51 & 46.25 & 55 \\
& & $8\times 8$ & 161 & 16.56 & 33 & 24.67 & 42 & 27.82 & 46 & 29.72 & 47 & 31.05 & 52 \\
& & $16\times 16$ & 705 & 11.61 & 31 & 17.93 & 39 & 20.64 & 42 & 22.37 & 47 & 23.63 & 50 \\
& & $32\times 32$ & 2\,945 & 8.11 & 26 & 13.67 & 36 & 16.18 & 40 & 17.84 & 44 & 19.07 & 48 \\
\cline{2-14}
& \multirow{4}{*}{\STAB{\rotatebox[origin=c]{90}{RGDSW}}}
& $4\times 4$ & 9 & 23.12 & 34 & 35.58 & 43 & 39.40 & 48 & 41.34 & 52 & 42.49 & 54 \\
& & $8\times 8$ & 49 & 16.18 & 34 & 23.90 & 42 & 26.83 & 44 & 28.61 & 49 & 29.86 & 53 \\
& & $16\times 16$ & 225 & 11.58 & 31 & 17.84 & 39 & 20.49 & 43 & 22.19 & 46 & 23.43 & 49 \\
& & $32\times 32$ & 961 & 8.02 & 25 & 13.63 & 35 & 16.17 & 39 & 17.84 & 43 & 19.07 & 47 \\
\hline
\hline
\multirow{8}{*}{\STAB{\rotatebox[origin=c]{90}{VEM}}}
& \multirow{4}{*}{\STAB{\rotatebox[origin=c]{90}{GDSW}}}
& $4\times 4$ & 33 & 25.47 & 34 & 44.60 & 50 & 47.54 & 55 & 48.95 & 56 & 50.39 & 56 \\
& & $8\times 8$ & 161 & 16.80 & 34 & 27.02 & 46 & 31.16 & 52 & 33.68 & 54 & 35.33 & 56 \\
& & $16\times 16$ & 705 & 11.86 & 31 & 19.44 & 42 & 23.50 & 50 & 26.10 & 53 & 27.84 & 55 \\
& & $32\times 32$ & 2\,945 & 8.33 & 27 & 14.81 & 39 & 18.85 & 48 & 21.48 & 50 & 23.24 & 53 \\
\cline{2-14}
& \multirow{4}{*}{\STAB{\rotatebox[origin=c]{90}{RGDSW}}}
& $4\times 4$ & 9 & 23.50 & 34 & 40.98 & 48 & 44.37 & 55 & 45.40 & 56 & 46.77 & 56 \\
& & $8\times 8$ & 49 & 16.48 & 34 & 26.27 & 47 & 30.41 & 54 & 32.78 & 55 & 34.36 & 56 \\
& & $16\times 16$ & 225 & 11.84 & 31 & 19.36 & 43 & 23.43 & 51 & 25.99 & 52 & 27.70 & 54 \\
& & $32\times 32$ & 961 & 8.25 & 26 & 14.80 & 39 & 18.86 & 47 & 21.48 & 49 & 23.24 & 52 \\
\hline
\end{tabular}
\end{adjustbox}
\end{table}

\section{Numerical results}
\label{sec:results}
All computations in this section use the analytical solution
\[
u(x,y)=\sin(\pi x)\sin(\pi y)
\]
on the unit square.
The experiments are carried out with an in-house \textsc{Matlab} R2025b implementation containing the FEM and VEM assemblers, the GDSW and RGDSW coarse-space construction, and the PCG solver. The VEM implementation is adapted from \cite{herrera_highorder_2d_vem_2023,sutton_vem_50lines_2017}.
PCG is stopped when the Euclidean norm of the unpreconditioned residual has been reduced by $10^{-8}$, and condition numbers are estimated by a Lanczos procedure.
Polygonal meshes are generated with PolyMesher~\cite{talischi_polymesher_2012} starting from randomly placed seeds. A Voronoi tessellation of the unit square is then computed and regularized by $300$ Lloyd iterations~\cite{du_centroidal_1999,lloyd_least_1982}, which produces well-shaped elements of similar size; see \Cref{fig:2d_lloyd_mesh}. We refer to this mesh as Lloyd-regularized mesh.

We clarify that the ``Coarse'' column in all tables reports the coarse space dimension for $k=1$. On polygonal meshes, short subdomain edges may be identified with vertices in the interface decomposition, causing small variations in the coarse space dimension when going from $k=1$ to $k=2$. For the structured triangular and quadrilateral meshes in \Cref{tab:2d_160_kappa,tab:2d_160_kappa_tri}, this does not occur and the coarse space dimension is independent of $k$.

\begin{figure}[t]
\centering
\includegraphics[width=\textwidth]{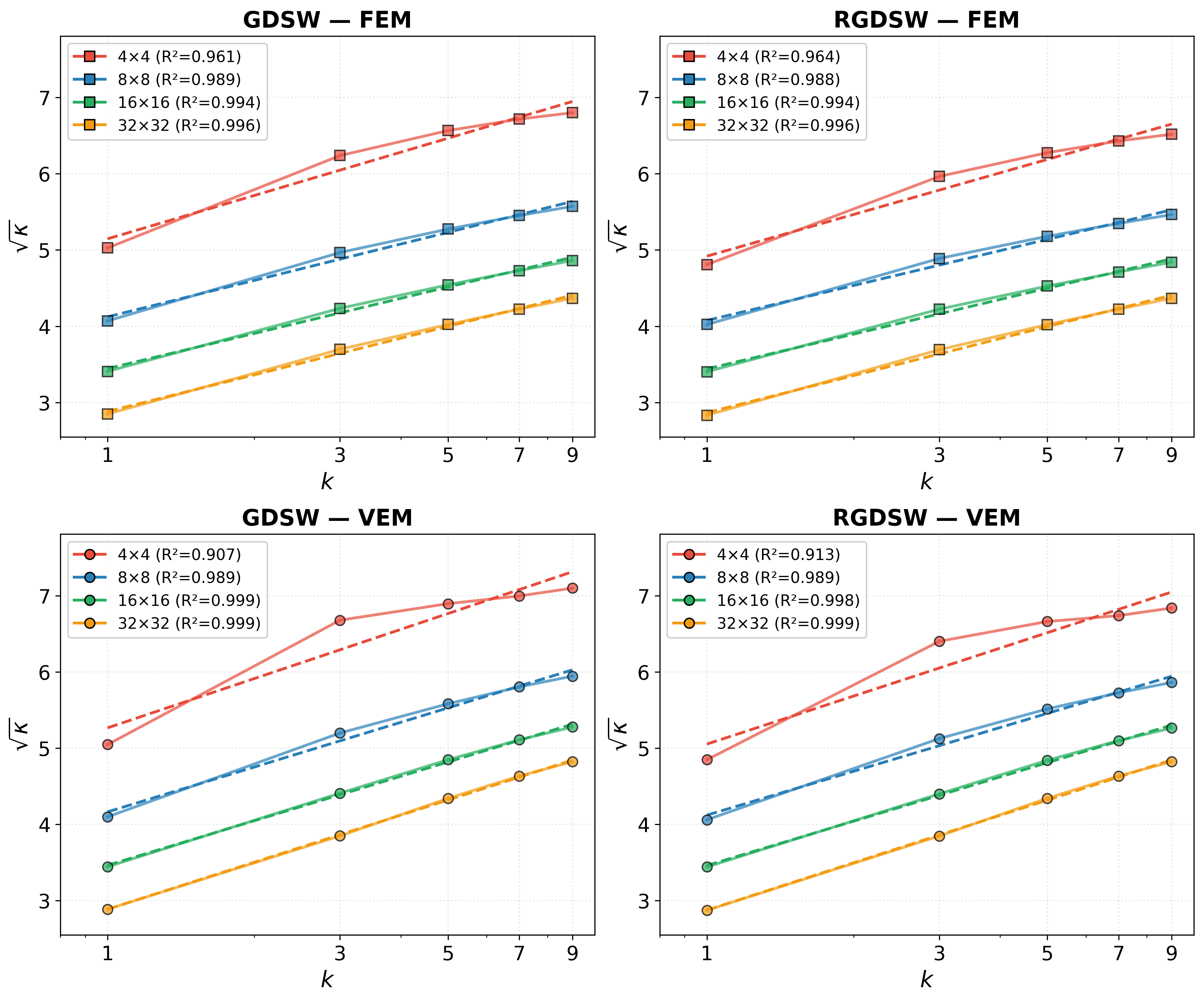}
\vspace{-0.7cm}
\caption{Condition numbers $\sqrt{\kappa}$ versus polynomial degree $k$ for VEM and $\mathbb{Q}_k$-FEM on a $160\times 160$ square mesh with overlap 1. Top row: FEM; bottom row: VEM. Left column: GDSW; right column: RGDSW. Each panel shows results for subdomain layouts $4\times 4$, $8\times 8$, $16\times 16$, $32\times 32$. Dashed lines indicate the fit $\kappa \sim (1+\log(H/h))(1+\log(Hk^2/h))$.}\label{fig:2d_vem_fem_kappa}
\end{figure}

\Cref{tab:2d_160_kappa,tab:2d_160_kappa_tri} report condition number estimates and PCG iteration counts for $\mathbb{P}_k$-FEM on structured triangular meshes (obtained by fixed-diagonal subdivision) and $\mathbb{Q}_k$-FEM on structured quadrilateral meshes, together with VEM results, all on a $160\times 160$ base mesh with overlap 1, for polynomial degrees $k=1,3,5,7,9$ and subdomain layouts $4\times 4$, $8\times 8$, $16\times 16$, and $32\times 32$. For the VEM discretization we use d-recipe stabilization with an orthogonal polynomial basis. Reading the tables vertically for a fixed polynomial degree, both GDSW and RGDSW exhibit the expected behavior: condition numbers decrease as the number of subdomains increases (i.e., as the ratio $H/h$ decreases), and the iteration counts remain moderate throughout. Reading the tables horizontally, increasing the polynomial degree from $k=1$ to $k=9$ leads to a gradual growth in condition numbers, consistent with the $(1+\log(H/h))(1+\log(Hk^2/h))$ dependency predicted by the theory for high-order FEM. For $k=1$, the FEM and VEM condition numbers are very similar, while at higher polynomial degrees, VEM produces moderately larger condition numbers than FEM.

The RGDSW coarse space achieves a substantial reduction in coarse problem dimension compared to GDSW by a factor of about three, while maintaining comparable condition numbers and iteration counts. At larger subdomain counts, the GDSW and RGDSW condition numbers converge. For the $32\times 32$ layout, the difference is negligible for both FEM and VEM. \Cref{fig:2d_vem_fem_kappa} visualizes the condition numbers $\sqrt{\kappa}$ versus the polynomial degree $k$ for both discretizations, with dashed lines indicating the fit $\kappa\sim (1+\log(H/h))(1+\log(Hk^2/h))$. The data points for both FEM and VEM closely follow this logarithmic scaling trend.

As an additional study on an unstructured polygonal mesh, \Cref{tab:2d_cvt_lloyd300_vem_kappa} reports the same PCG/Lanczos metrics for VEM on the Lloyd-regularized mesh with $50\,080$ polygonal elements ($\approx\!100\,000$ to $\approx\!3\,100\,000$ \ac{DOFs} for $k=1$ to $k=9$), using overlap 1 and METIS partitioning; see \Cref{fig:2d_lloyd_mesh}. The condition numbers on this unstructured mesh are noticeably higher than on the structured mesh, yet the qualitative behavior is similar. Both GDSW and RGDSW maintain robust scalability across subdomain counts ranging from $16$ to $1\,024$, and the condition numbers grow only mildly with $k$, with only slight increases beyond $k=5$. The RGDSW coarse space again achieves a significant reduction in dimension compared to GDSW while delivering nearly identical iteration counts. \Cref{fig:2d_lloyd_vem_kappa} visualizes these condition numbers and shows that the logarithmic scaling is observed also on general polygonal meshes, although the fit is less close than on the structured square mesh.

\begin{table}[t]
\centering
\caption{Condition numbers and PCG iteration counts for VEM on the Lloyd-regularized mesh, overlap 1, d-recipe stabilization, orthogonal polynomial basis, and METIS partitioning.}\label{tab:2d_cvt_lloyd300_vem_kappa}
\begin{adjustbox}{width=0.98\textwidth}
\begin{tabular}{|c|r|r||r|r||r|r||r|r||r|r||r|r|}
\cline{2-13}
\multicolumn{1}{c|}{} & & nDofs & \multicolumn{2}{c||}{$100\,246$} & \multicolumn{2}{c||}{$551\,136$} & \multicolumn{2}{c||}{$1\,202\,346$} & \multicolumn{2}{c||}{$2\,053\,876$} & \multicolumn{2}{c|}{$3\,105\,726$} \\
\cline{2-13}
\multicolumn{1}{c|}{} & nSub & Coarse & \multicolumn{2}{c||}{$k=1$} & \multicolumn{2}{c||}{$k=3$} & \multicolumn{2}{c||}{$k=5$} & \multicolumn{2}{c||}{$k=7$} & \multicolumn{2}{c|}{$k=9$} \\
\cline{2-13}
\multicolumn{1}{c|}{} &  &  & $\kappa$ & it & $\kappa$ & it & $\kappa$ & it & $\kappa$ & it & $\kappa$ & it \\
\cline{2-13}
\noalign{\vskip 0.6ex}
\cline{1-13}
\multirow{5}{*}{\STAB{\rotatebox[origin=c]{90}{GDSW}}}
& 16 & 51 & 32.44 & 58 & 56.75 & 60 & 63.25 & 64 & 64.59 & 64 & 65.28 & 65 \\
& 64 & 255 & 22.31 & 51 & 38.38 & 51 & 43.69 & 54 & 44.50 & 54 & 45.06 & 55 \\
& 256 & 1\,155 & 18.21 & 46 & 31.39 & 46 & 36.71 & 50 & 37.50 & 51 & 38.09 & 51 \\
& 512 & 2\,383 & 15.55 & 42 & 27.10 & 43 & 32.55 & 47 & 33.75 & 48 & 34.46 & 48 \\
& 1\,024 & 4\,865 & 13.71 & 39 & 21.94 & 39 & 26.97 & 43 & 27.32 & 43 & 27.28 & 45 \\
\hline
\hline
\multirow{5}{*}{\STAB{\rotatebox[origin=c]{90}{RGDSW}}}
& 16 & 18 & 31.67 & 56 & 54.91 & 58 & 61.52 & 62 & 62.81 & 64 & 63.34 & 64 \\
& 64 & 96 & 22.50 & 51 & 38.68 & 51 & 43.87 & 55 & 44.61 & 56 & 45.11 & 57 \\
& 256 & 450 & 17.26 & 45 & 28.80 & 45 & 34.41 & 49 & 35.12 & 50 & 35.45 & 51 \\
& 512 & 936 & 14.78 & 41 & 26.60 & 42 & 30.35 & 46 & 31.42 & 47 & 31.85 & 47 \\
& 1\,024 & 1\,921 & 12.72 & 38 & 22.30 & 40 & 26.86 & 44 & 27.07 & 44 & 27.43 & 45 \\
\hline
\end{tabular}
\end{adjustbox}
\end{table}

\begin{figure}[t]
\centering
\includegraphics[width=\textwidth]{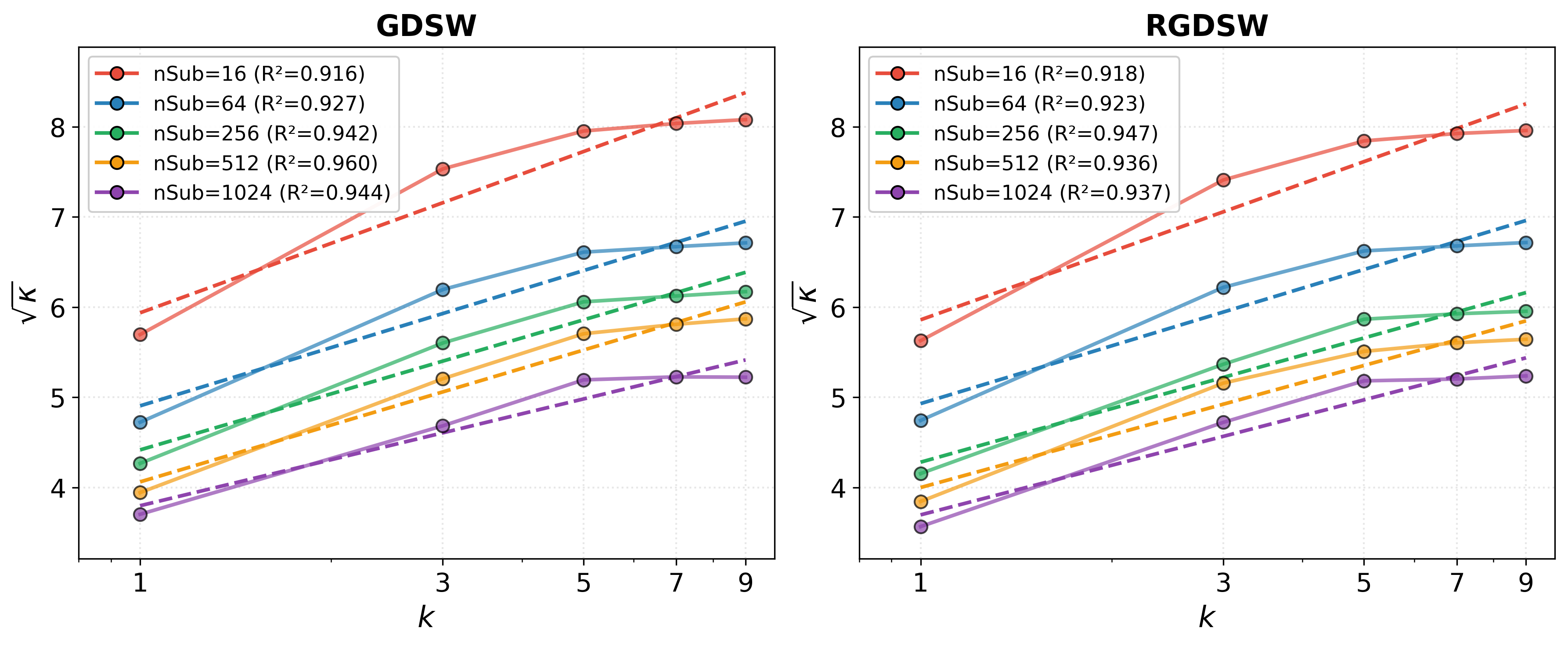}
\vspace{-0.7cm}
\caption{Condition numbers $\sqrt{\kappa}$ versus polynomial degree $k$ for VEM on the Lloyd-regularized mesh, overlap 1. Left: GDSW; right: RGDSW. Each panel shows results for nSub$=16,64,256,512,1024$. Dashed lines indicate the fit $\kappa \sim (1+\log(H/h))(1+\log(Hk^2/h))$.}\label{fig:2d_lloyd_vem_kappa}
\end{figure}

\section{Conclusion}

This paper analyzes GDSW and RGDSW two-level overlapping Schwarz preconditioners for two-dimensional high-order FEM and VEM Poisson discretizations.
For conforming $\mathbb{P}_k$ and $\mathbb{Q}_k$ finite element spaces, the main theoretical result is a condition number estimate in which the dependence on the polynomial degree enters explicitly.

The numerical results confirm scalability and $k$-robustness of both GDSW and RGDSW across polynomial degrees $k=1,\ldots,9$ and subdomain counts ranging from $16$ to $1\,024$. Condition numbers grow only slightly with $k$, and they are consistent with the logarithmic dependence predicted by the theory. For $k=1$, FEM and VEM condition numbers are almost identical; at higher polynomial degrees, VEM can produce slightly larger condition numbers than FEM, while the qualitative behavior remains the same. The RGDSW coarse space achieves a significant reduction in coarse-space dimension compared to GDSW, while delivering nearly identical iteration counts.

The extension of the analysis to high-order VEM, and in particular a proof that is uniform in the polynomial degree, remains open.
Further developments will also include three-dimensional settings and additional model problems.

\section*{Acknowledgments}
This project has received funding from the Deutsche Forschungsgemeinschaft (DFG) as part of the Forschungsgruppe (Research Unit) 5134 ``Solidification Cracks During Laser Beam Welding -- High Performance Computing for High Performance Processing'' under DFG project number 434946896. The authors gratefully acknowledge the scientific support and HPC resources provided by the Erlangen National High Performance Computing Center (NHR@FAU) of the Friedrich-Alexander-Universit\"at Erlangen-N\"urnberg (FAU) under the NHR project k109be10. NHR funding is provided by federal and Bavarian state authorities. NHR@FAU hardware is partially funded by the German Research Foundation (DFG) - 440719683.

\section*{Author contributions and AI usage} All authors are listed in alphabetical order. The contributions are specified below using the CRediT taxonomy. TB: Conceptualization, Methodology, Writing - original draft, review \& editing. AK: Conceptualization, Methodology, Formal Analysis, Writing - review \& editing. ML: Conceptualization, Methodology, Writing - review \& editing. AW: Conceptualization, Methodology, Formal Analysis, Data Curation, Software, Writing - original draft.
Models were used to assist with formatting and improving the clarity of portions of the text, as well as to help with the development and debugging of the \textsc{Matlab} implementation. The authors assume full responsibility for all content.

\bibliographystyle{siamplain}
\bibliography{refer}

\end{document}